\documentclass[a4paper,11pt]{article}
\usepackage[utf8x]{inputenc}

\usepackage{amssymb}
\usepackage{amsmath}
\usepackage{amsthm}
\usepackage[mathscr]{euscript}
\usepackage[normalem]{ulem}
\usepackage[colorlinks=true]{hyperref}
\usepackage[usenames,dvipsnames]{xcolor}
\usepackage{tikz}
\usepackage{appendix}
\usetikzlibrary{arrows,angles}

\numberwithin{equation}{section}

\theoremstyle{definition}

\newtheorem{dfn}{Definition}[section]

\newtheorem{thm}[dfn]{Theorem}
\newtheorem{pop}[dfn]{Proposition}
\newtheorem{lem}[dfn]{Lemma}
\newtheorem{cor}[dfn]{Corollary}
\newtheorem{cmkl}[dfn]{Folklore}

\newtheorem{manualtheoreminner}{Theorem}
\newenvironment{introtheorem}[1]{%
  \renewcommand\themanualtheoreminner{#1}%
  \manualtheoreminner
}{\endmanualtheoreminner}

\newtheorem{manualcorinner}{Corollary}
\newenvironment{introcor}[1]{%
  \renewcommand\themanualcorinner{#1}%
  \manualcorinner
}{\endmanualcorinner}

\theoremstyle{remark}
\newtheorem{rem}[dfn]{Remark}
\newtheorem{ex}[dfn]{Example}

\newcommand{\R}{\mathbb R}

\newcommand{\lm}[1]{\mathbb{L}^2(#1)}

\usepackage[inline]{enumitem}
\newcommand{\enumlabelformat}{\roman}
\newcommand{\enumlabelfont}[1]{#1}
\newlength{\thelabelsep}
\setlist{labelsep=\thelabelsep}
\setlist[enumerate,1]{font=\enumlabelfont,label=(\enumlabelformat*),leftmargin=2.5em}
\setlist[itemize]{leftmargin=2.5em,label=$-$}

\newcounter{inlineenum}
\renewcommand{\theinlineenum}{\enumlabelformat{inlineenum}}

\newcommand{\mb}[1]{\mathbb{#1}}
\newcommand{\diff}{\ensuremath{\mathrm{d}}}
\DeclareMathOperator{\id}{id}
\DeclareMathOperator{\oldAdS}{AdS}
\newcommand{\AdS}{\ensuremath{\oldAdS} }
\newcommand{\AdSn}{\ensuremath{\oldAdS}}
\DeclareMathOperator{\arcosh}{arcosh}

\newcommand{\LpLS}{Lorentzian pre-length space }
\newcommand{\LpLSn}{Lorentzian pre-length space}
\newcommand{\ma}{\measuredangle}

\makeatletter
\let\save@mathaccent\mathaccent
\newcommand*\if@single[3]{%
  \setbox0\hbox{${\mathaccent"0362{#1}}^H$}%
  \setbox2\hbox{${\mathaccent"0362{\kern0pt#1}}^H$}%
  \ifdim\ht0=\ht2 #3\else #2\fi
  }
\newcommand*\rel@kern[1]{\kern#1\dimexpr\macc@kerna}
\newcommand*\widebar[1]{\@ifnextchar^{{\wide@bar{#1}{0}}}{\wide@bar{#1}{1}}}
\newcommand*\wide@bar[2]{\if@single{#1}{\wide@bar@{#1}{#2}{1}}{\wide@bar@{#1}{#2}{2}}}
\newcommand*\wide@bar@[3]{%
  \begingroup
  \def\mathaccent##1##2{%
    \let\mathaccent\save@mathaccent
    \if#32 \let\macc@nucleus\first@char \fi
    \setbox\z@\hbox{$\macc@style{\macc@nucleus}_{}$}%
    \setbox\tw@\hbox{$\macc@style{\macc@nucleus}{}_{}$}%
    \dimen@\wd\tw@
    \advance\dimen@-\wd\z@
    \divide\dimen@ 3
    \@tempdima\wd\tw@
    \advance\@tempdima-\scriptspace
    \divide\@tempdima 10
    \advance\dimen@-\@tempdima
    \ifdim\dimen@>\z@ \dimen@0pt\fi
    \rel@kern{0.6}\kern-\dimen@
    \if#31
      \overline{\rel@kern{-0.6}\kern\dimen@\macc@nucleus\rel@kern{0.4}\kern\dimen@}%
      \advance\dimen@0.4\dimexpr\macc@kerna
      \let\final@kern#2%
      \ifdim\dimen@<\z@ \let\final@kern1\fi
      \if\final@kern1 \kern-\dimen@\fi
    \else
      \overline{\rel@kern{-0.6}\kern\dimen@#1}%
    \fi
  }%
  \macc@depth\@ne
  \let\math@bgroup\@empty \let\math@egroup\macc@set@skewchar
  \mathsurround\z@ \frozen@everymath{\mathgroup\macc@group\relax}%
  \macc@set@skewchar\relax
  \let\mathaccentV\macc@nested@a
  \if#31
    \macc@nested@a\relax111{#1}%
  \else
    \def\gobble@till@marker##1\endmarker{}%
    \futurelet\first@char\gobble@till@marker#1\endmarker
    \ifcat\noexpand\first@char A\else
      \def\first@char{}%
    \fi
    \macc@nested@a\relax111{\first@char}%
  \fi
  \endgroup
}
\makeatother

\title{A Bonnet-Myers rigidity theorem for globally hyperbolic Lorentzian length spaces}

\author{Tobias Beran\footnote{Department of Mathematics, University of Vienna, Oskar-Morgenstern-Platz 1, 1090 Wien, Austria, \newline
tobias.beran@univie.ac.at}\\ 
}

\begin{document}

\date{\today}


\maketitle

\begin{abstract}
We prove a synthetic Bonnet-Myers rigidity theorem for globally hyperbolic Lorentzian length spaces with global curvature bounded below by $K<0$ and an open distance realizer of length $L=\frac{\pi}{\sqrt{|K|}}$: It states that the space necessarily is a warped product with warping function $\cos:(-\frac{\pi}{2},\frac{\pi}{2})\to\mb{R}_+$. From this, one also sees that a globally hyperbolic spacetime with curvature bounded above by $K<0$ and an open distance realizer of length $L=\frac{\pi}{\sqrt{|K|}}$ is a warped product with warping function $\cos$.
\vspace{1em}

\noindent
\emph{Keywords:} Lorentzian length spaces, synthetic curvature bounds, warped product spaces
\medskip

\noindent
\emph{MSC2020:} 53C50, 53C23, 53B30

\end{abstract}
\newpage
\tableofcontents
\newpage

\section{Introduction}\label{sec:intro}
In the setting of Riemannian manifolds, the Bonnet-Myers theorem states that a complete Riemannian manifold $M$ with sectional curvature bounded below by some $K>0$ has diameter less than $\frac{\pi}{\sqrt{K}}$. As a Riemannian manifold with sectional curvature bounded below by $K$ has Ricci curvature bounded below by $(n-1)K$, it is a special case of the Myers theorem: It states that a complete $n$-dimensional Riemannian manifold $M$ with Ricci curvature bounded below by some $(n-1)K>0$ has diameter less than $\frac{\pi}{\sqrt{K}}$ (see \cite[Thm.\ I]{myersTheorem}).

The Bonnet-Myers theorem has been generalized to complete geodesic metric spaces, see \cite[Thm.\ 3.6]{BGP}: 
\begin{thm}
Let $(X,d)$ be a complete metric space with Alexandrov curvature bounded below by $K>0$. 

Then its diameter is bounded above by $\frac{\pi}{\sqrt{K}}$
\end{thm}

The corresponding rigidity theorems talk about the case where the bound on the diameter is achieved by a minimizing geodesic. A Myers rigidity result was proven in \cite[Thm.\ 3.1]{ChengMyersRigidity}: 
\begin{thm}
Let $M$ be a complete Riemannian manifold of dimension $n$ with Ricci curvature bounded below by $(n-1)K$ (for some $K>0$) and diameter equal to $\frac{\pi}{\sqrt{K}}$. Then $M$ is isometric to $\frac{1}{\sqrt{K}}S^n$.
\end{thm}

As sectional curvature bounds imply a Ricci curvature bound, this can be seen as a Bonnet-Myers rigidity theorem, assuming sectional curvature bounds instead of Ricci curvature bounds. 

But when removing the completeness or replacing complete Riemannian manifolds by metric length spaces, one immediately faces a counterexample:

\begin{ex}
Let $r>0$ and set $X$ to be the warped product $X= [-\frac{\pi}{2},\frac{\pi}{2}]\times_{\cos} rS^1$. Note at $t=\pm\frac{\pi}{2}$, $\cos(t)$ is $0$, so we need to take the quotient identifying $\{-\frac{\pi}{2}\}\times rS^1$ and $\{+\frac{\pi}{2}\}\times rS^1$ to a point, respectively. These points are called the \emph{poles} and will be denoted by $\pm\frac{\pi}{2}$.

Only for $r= 1$ this is a Riemannian manifold ($X\cong S^2$), otherwise only $\tilde{X}=X\setminus \{\pm\frac{\pi}{2}\}$ is a Riemannian manifold, and the neighbourhood of $\pm\frac{\pi}{2}$ infinitesimally looks like a cone (to be precise $Cone(rS^1)$, a ``conical singularity''). $\tilde{X}$ has Ricci curvature bounded below by $1$, and synthetically, only for $r\leq 1$ the Ricci curvature of $X$ is bounded below by $1$ at $\pm\frac{\pi}{2}$. 

It has diameter $\pi$, thus $X$ is a counterexample for the length space case (for $r<1$ and using synthetic Ricci curvature bounds, see \cite[Thm.\ 1.4]{RicciSphericalCones}) and $X\setminus\{\pm\frac{\pi}{2}\}$ is a counterexample for the incomplete Riemannian manifold case (for any $r\neq 1$).
\end{ex}
With additional assumptions, one can recover the rigidity result for length spaces, see \cite[Thm.\ 1.4 and Cor.\ 1.6]{kettererSynthMyers}.

Without additional assumptions and assuming Alexandrov curvature bounds, a well-known result usually attributed to Grove and Petersen (see \cite[Lem.\ 29]{ref_to_metricBRig} for a statement and \cite[Lem.\ 2.5]{refUse1_to_metricBRig} for a reference) states that the following holds:
\begin{cmkl}
Let $X$ be a complete geodesic metric space with global Alexandrov curvature bounded below by $K=1$ and assume there exist $p,q\in X$ with $d(p,q)=\pi$. Then $X$ is a spherical suspension of a complete geodesic metric space $Y$ with global Alexandrov curvature bounded below by $K=1$, i.e.\ $X$ is the warped product $[-\frac{\pi}{2},\frac{\pi}{2}]\times_{\cos} Y$.
\end{cmkl}
A proof of this can be derived from the proof in this paper.

\medskip

Similarly, in the setting of Lorentzian manifolds, there is a Myers style theorem which can be found in \cite[Thm.\ 3.4.2]{treudeDipl}:
\begin{thm}
Let $M$ be a globally hyperbolic time-oriented $n$-dimensional Lorentzian manifold with  timelike Ricci curvature bounded from below by $K>0$ (i.e.\ $Ric(v, v) \geq (n - 1)K$).

Then its timelike diameter is bounded above by $\frac{\pi}{\sqrt K}$.
\end{thm}
In the case of synthetic Ricci curvature bounds, there is a synthetic Myers theorem (see \cite[Prop.\ 5.10]{CavalettiMondino}, \cite[Thm.\ 5.5]{McCannVariableCurvatureBounds}, \cite[Cor.\ 3.14]{BraunRenyiEntropy}), and this has recently been used to solve the low regularity manifold Myers theorem (see \cite[Cor.\ 3.10]{BraunRicciInLowRegularity}). 
To the best of our knowledge, there is no known result giving strict Myers rigidity in both the smooth and synthetic Lorentzian Ricci comparison (i.e.\ a result implying that the space is higher dimensional anti-deSitter space).

As timelike sectional curvature bounds imply timelike Ricci curvature bounds in the smooth case, the spacetime Myers style theorem \cite[Thm.\ 3.4.2]{treudeDipl} implies a sectional curvature version. 

In the setting of synthetic sectional curvature bounds (see \cite{saemLorLen}), there also is a Bonnet-Myers theorem \cite[Thm.\ 4.11, Rem.\ 4.12]{alexPatchAndBonnet}:
\begin{thm}[Synthetic Lorentzian Bonnet-Myers]
\label{thm: lor meyers}
Let $X$ be a strongly causal, locally causally closed, regular, and geodesic Lorentzian pre-length space which has global curvature bounded below by $K$. Assume $K<0$. Assume that $X$ possesses the following non-degeneracy condition: for each pair of points $x\ll z$ in $X$ we find $y \in X$ such that $\Delta(x,y,z)$ is a non-degenerate timelike triangle.
 
Then $\tau(p,q)>D_K$ can only hold if $\tau(p,q)=+\infty$.
\end{thm}
Note the non-degeneracy condition is there to ensure the space is not locally one-dimensional.

We will prove a rigidity result for the synthetic Lorentzian Bonnet-Myers theorem, stating the following:
\begin{introtheorem}{\ref{thm:globalsplitting}}
Let $(X,d,\ll,\leq,\tau)$ be a connected, regularly localisable, globally hyperbolic Lorentzian length space with proper metric $d$ and global timelike curvature bounded below by $K=-1$ satisfying timelike geodesic prolongation and containing a $\tau$-arclength parametrized distance realizer $\gamma:(-\frac{\pi}{2},\frac{\pi}{2}) \to X$. Assume that for each pair of points $x\ll z$ in $X$ we find $y \in X$ such that $\Delta(x,y,z)$ is a non-degenerate timelike triangle.

Then there is a proper (hence complete), strictly intrinsic metric space  $S$ such that the Lorentzian warped product $(-\frac{\pi}{2},\frac{\pi}{2}) \times_{\cos} S$ is a path-connected, regularly localisable, globally hyperbolic Lorentzian length space and there is a map $f:(-\frac{\pi}{2},\frac{\pi}{2}) \times_{\cos} S \to I(\gamma)$ which is a $\tau$- and $\leq$-preserving homeomorphism.
\end{introtheorem}

Additionally, we get:

\begin{introcor}{\ref{cor:formCauchySets}}
The sets $S_t:=f(\{t\} \times S)$ are Cauchy sets in $X$ that are all homeomorphic to $S$. Moreover, let $\varphi:(-\frac{\pi}{2},\frac{\pi}{2})\to\R$ be a monotonically increasing bijection, then the map $\varphi\circ pr_1 \circ f^{-1}$ is a Cauchy time function. Moreover, all Cauchy sets in $X$ are homeomorphic to $S$.
\end{introcor}

\begin{introcor}{\ref{cor:SAlexCurv}}
$(S,d_S)$ has Alexandrov curvature bounded below by $-1$.
\end{introcor}

\begin{introcor}{\ref{cor:spacetimes}}
Let $(M,g)$ be a connected globally hyperbolic spacetime of dimension $n\geq 2$ with smooth timelike sectional curvature bounded above\footnote{This might seem the wrong direction, but this is due to the signature of Lorentzian metrics.} by $K=-1$ and containing a timelike distance realizer of length $\pi$. Furthermore assume along each timelike distance realizer, there are no conjugate points of degree $n-1$. Then there is a spacelike Cauchy surface $S$ in $M$, endowed by a metric from the Riemannian metric $g|_S$ and a map $f:(-\frac{\pi}{2},\frac{\pi}{2})\times_{\cos} S\to M$ which is an isometry and a $C^1$-diffeomorphism, restricting to the identity $\{0\}\times S\to S$.
\end{introcor}

\bigskip

The proof of the Lorentzian Bonnet-Myers rigidity theorem will roughly follow the guide given by the proof of the splitting theorem for Lorentzian length spaces with non-negative timelike curvature, see \cite{triSplitting}, in particular the outline is in parts very similar to there.

The paper is organised as follows: In section \ref{sec:LLS}, we review the concepts of timelike curvature bounds, angles and comparison angles, with their basic implications. In section \ref{sec:warpedproducts}, we give an alternate description of warped product spaces which were first introduced in \cite{generalizedCones} which is better suited for our needs. In subsection \ref{subsec:warpedproducts:AdSasWarpedProd}, we give a description of a part of Anti-deSitter spacetime as such a warped product, and define \AdSn-suspensions. In subsection \ref{subsec:lines:lines}, we discuss \AdSn-lines and asymptotes (including the heart of the proof, the stacking principle), and in subsection \ref{subsec:lines:parallel}, we introduce the concept of \AdSn-parallelism and relate asymptotes with parallelism. Finally, in section \ref{sec:proof} we piece the parts together to a proof and further implications.

\subsection*{Notation and conventions}
\label{subsec:notationconventions}

Let us collect some notation and conventions that will be used throughout the paper.

A \emph{proper} metric space $(X,d)$ is a metric space such that all closed balls are compact. 

For Lorentzian manifolds, we choose the signature $(-,+,\cdots,+)$ (but this only has side effects and is not directly used in this paper). 
For \LpLSn, as we work with strongly causal spaces, geodesics can be defined either by the definition for localizable spaces or by requiring them to be covered by domains which are globally distance realizing. Mostly, we will only use distance realizers anyway. \AdS is two-dimensional anti-deSitter spacetime, which will be introduced in Def.\ \ref{def:AdS}, and $\AdS'$ is the warped product inside it, see Lem.\ \ref{lem:AdS'}. $\overline{\tau}$ denotes the time separation on $\AdS'$ (and $\AdS$). $\tilde{\ma}_x(y,z)=\tilde{\ma}^{-1}_x(y,z)$ is the comparison angle calculated in \AdSn. 

\section{Preliminaries}
\subsection{Basic theory of Lorentzian (pre-)length spaces}\label{sec:LLS}
We follow the definitions of Lorentzian (pre-)length spaces in \cite{triSplitting}, deviating only in curvature bounds and not using timelike geodesic completeness and extensibility. One can also find these definitions in \cite{saemLorLen}. For self-containedness, I include the basic definitions here:
\begin{dfn}\label{def:LpLS}
A \emph{causal space} is a set $X$ together with a reflexive and transitive relation $\leq$ and a transitive relation $\ll$ contained in $\leq$. We call $\leq$ the \emph{causal} relation and $\ll$ the \emph{timelike} relation. For $x\leq y$ we say $x$ is causally before $y$, similarly if $x\ll y$ we say $x$ is timelike before $y$.

A \emph{Lorentzian pre-length space} is a causal space $(X,\ll,\leq)$ together with a metric $d$ on $X$ and a lower semicontinuous function $\tau:X\times X\to[0,\infty]$, the \emph{time separation function}, satisfying:
\begin{itemize}
\item timelikeness: $\tau(x,y)>0$ if and only if $x\ll y$
\item the reverse triangle inequality: if $x\leq y\leq z$, $\tau(x,z)\geq\tau(x,y)+\tau(y,z)$
\end{itemize}
\end{dfn}

\subsection{Timelike curvature bounds}\label{subsec:LLS:tlCurvBds}

The timelike curvature bounds were first introduced in \cite{saemLorLen}, and slightly modified in \cite{alexPatchAndBonnet}. To describe timelike curvature bounds, we will compare certain distances to distances in comparison spaces: the \emph{Lorentzian model spaces} $\lm{K}$ of constant sectional curvature $K$. 

We will work with timelike curvature bounds below by $K<0$ exclusively. As other negative curvature bounds follow easily by scaling the space, we only need $K=-1$. For self-containedness, we include the case $K<0$ here. For more details on the other cases, see \cite[Def.\ 4.5]{saemLorLen}.

A word of warning though: In the Lorentzian case (in our signature) having sectional curvature bounded below by $K$ as in \cite[Chp.\ 1]{semiRiemAlexBounds} actually requires an \emph{upper} bound on the sectional curvature of timelike planes. Our timelike curvature bounds stem from bounds of sectional curvature for timelike planes and thus \emph{the inequalities intuitively point in the wrong direction}, e.g.\ flat Minkowski space does not have timelike curvature bounded below by $K=-1$, but above. Keeping the inequalities in this direction makes the behaviour mostly match the metric case (e.g.\ lower curvature bounds prohibit branching, some lower curvature bounds make the diameter finite in some way). One could of course then change the sign of $K$, but following \cite{saemLorLen} we will not do that.

\begin{dfn}\label{def:AdS}
Let $(\mb{R}^{1,2},b)$ be the $3$-dimensional semi-Riemannian space of signature $-,-,+$. We define \emph{anti-deSitter space} $\lm{-1}=\AdS$ as the universal cover of the set $\{p\in\mb{R}^{1,2}:b(p,p)=-1\}$. Equipping the tangent space with the restriction of $b$ makes this a $2$-dimensional Lorentzian manifold, with appropriately chosen time orientation. Making this space into a Lorentzian pre-length space, we get the anti-deSitter time separation function $\bar{\tau}$. 
Scaling this space, we get the Lorentzian model space of constant sectional curvature $K<0$: $\lm{K}=(\AdS,\frac{1}{\sqrt{-K}}\bar{\tau})$.

In any of the $\lm{K}$ ($K<0$) there are points of infinite $\tau$-distance. We define $D_{K}=\frac{\pi}{\sqrt{-K}}$, this is the maximal value $\tau$ will take before it becomes infinite. 
\end{dfn}

We call three points $x_1,x_2,x_3 \in X$ that are timelike related together with maximisers between them a \emph{timelike triangle}. We denote such triangles by $\Delta(x_1,x_2,x_3)$. 

\begin{dfn}
Let $X$ be a \LpLS and $\Delta(x_1,x_2,x_3)$ be a timelike triangle. We say it satisfies the \emph{size bounds for $K$} if $\tau(x_i,x_j)<D_K$ (note by the timelike relations, this only needs to be checked for one pair). If they do, a timelike triangle $\Delta(\bar{x}_1,\bar{x}_2,\bar{x}_3)$ in $\lm{K}$ such that $\tau(x_i,x_j)=\bar{\tau}(\bar{x}_i,\bar{x}_j)$ is called a \emph{comparison triangle} for $\Delta(x_1,x_2,x_3)$. It always exists if the size bounds are satisfied and is unique up to isometries of $\lm{K}$.
\end{dfn}

\begin{dfn}[Timelike curvature bounds by triangle comparison]
\label{def:triComp}
Let $X$ be a \LpLSn. An open subset $U$ is called a timelike $(\geq K)$-comparison neighbourhood in the sense of \emph{triangle comparison} if:
\begin{enumerate}
\item $\tau$ is continuous on $(U\times U) \cap \tau^{-1}([0,D_K))$, and this set is open.
\item For all $x,y \in U$ with $x \ll y$ and $\tau(x,y) < D_K$ there exists a geodesic connecting them which is contained entirely in $U$.
\item Let $\Delta (x,y,z)$ be a timelike triangle in $U$ satisfying size bounds for $K$, 
with $p,q$ two points on the sides of $\Delta (x,y,z)$. Let 
$\bar\Delta(\bar{x}, \bar{y}, \bar{z})$ be a comparison triangle in $\lm{K}$ for $\Delta (x,y,z)$ and $\bar{p},\bar{q}$ comparison points for $p$ and $q$, respectively. Then
\begin{equation}
\tau(p,q) \leq \tau(\bar{p},\bar{q}).
\end{equation}
\end{enumerate}

We say $X$ has timelike curvature bounded below by $K$ (in the sense of \emph{triangle comparison}) if it is covered by timelike $(\geq K)$-comparison neighbourhoods (in the sense of \emph{triangle comparison}). 

We say $X$ has global timelike curvature bounded below by $K$ (in the sense of \emph{triangle comparison}) if $X$ itself is a $(\geq K)$-comparison neighbourhood (in the sense of \emph{triangle comparison}).
\end{dfn}

\begin{rem}[Continuous triangles vs.\ Lipschitz triangles]
\label{rem:contvsLipschitztriangles}
If $X$ is a globally hyperbolic Lorentzian length space with global timelike curvature bounded below / above by $K$, then in fact curvature comparison even holds for timelike triangles where the maximisers are only continuous: Indeed, suppose $\Delta:=\Delta_{C^0}(x,y,z)$ is a continuous timelike triangle, and let $p,q \in \Delta$. Due to the second condition of timelike curvature bounds, we find (Lipschitz) maximisers from the endpoints of $\Delta$ to $p,q$, respectively. The concatenations at $p$ resp.\ $q$ of two maximisers each are again maximisers because the sides on $\Delta$ are (continuous) maximisers. Hence we have realised a Lipschitz triangle $\Delta(x,y,z)$ with $p,q$ on its sides. (But this does not help to get Lipschitz maximisers if one would define the curvature bounds with continuous maximizers.)
\end{rem}

One of the most commonly used implications of lower (timelike) curvature bounds is the prohibition of branching of distance-realisers. A formulation of this result for Lorentzian pre-length spaces was first given in \cite[Thm.\ 4.12]{saemLorLen}. However, with the introduction of hyperbolic angles in \cite{AngLLS} it was possible to generalise this result by omitting some of the additional assumptions:

\begin{thm}[Timelike non-branching]\label{thm:non-branching}
Let $X$ be a strongly causal Lorentzian pre-length space with timelike curvature bounded below. Then timelike distance realisers cannot branch, i.e., if $\alpha, \beta: [-\varepsilon, \varepsilon] \to X$ are timelike distance realisers such that there exists $t_0 \in \R$ with $\alpha|_{[-\varepsilon,t_0]}=\beta|_{[-\varepsilon,t_0]}$, then $\alpha$ is a reparametrization of a part of $\beta$ or conversely.
\end{thm}
\begin{proof}
See \cite[Thm.\ 4.7]{AngLLS}.
\end{proof}

The non-branching of timelike distance realisers is a key property of spaces with lower curvature bounds and will appear in various forms in this proof of Bonnet-Myers rigidity.

\subsection{Angles and comparison angles}\label{subsec:LLS:angles}

Hyperbolic angles in Lorentzian pre-length spaces were introduced in \cite{AngLLS} and \cite{didierAngles}, where the latter puts a bigger focus on comparison results. We will follow the conventions of the former reference.

\begin{lem}[The law of cosines ($K=-1$)]\par
\label{lem:LOC}
Let $X=\AdS$ be anti-deSitter space and $x_1,x_2,x_3$ be three points which are timelike related (an (unordered) timelike triangle). Let $a_{ij}=\max(\bar{\tau}(x_i,x_j),\bar{\tau}(x_j,x_i))$ (note one of these is zero anyway). Let $\omega$ be the hyperbolic angle between the distance realizers $x_1x_2$ and $x_2x_3$ at $x_2$. Set $\sigma=1$ if $x_2$ is not a time endpoint of the triangle (i.e., $x_1\ll x_2\ll x_3$ or $x_3\ll x_2\ll x_1$) and $\sigma=-1$ if $x_2$ is a time endpoint of the triangle (i.e., $x_2\ll x_1,x_3$ or $x_1,x_3\ll x_2$). Then we have: 
\begin{equation*}
\cos(a_{13}) = \cos(a_{12}) \cos(a_{23}) - \sigma  \sin(a_{12}) \sin(a_{23})\cosh(\omega).
\end{equation*}
In particular, when only changing one side-length, the angle $\omega$ is a monotonically increasing function of the longest side-length and monotonically decreasing in the other side-lengths.
\end{lem}
\begin{proof}
See \cite[Appendix A]{AngLLS}.
\end{proof}

For $a_{ij}>0$ satisfying a reverse triangle inequality and choosing an appropriate $\sigma=\pm 1$, we can always solve this equation for $\omega$.

\begin{dfn}[Comparison angles]
Let $X$ be a \LpLS and $x_1,x_2,x_3$ three timelike related points. Let $\bar{x}_1,\bar{x}_2,\bar{x}_3\in \AdS$ be a comparison triangle for $x_1,x_2,x_3$. We define the \emph{comparison angle} $\tilde{\ma}_{x_2}^{-1}(x_1,x_3)$ as the hyperbolic angle between the straight lines $\bar{x}_1\bar{x}_2$ and $\bar{x}_2\bar{x}_3$ at $\bar{x}_2$. It can be calculated with the law of cosines using $a_{ij}=\max(\tau(x_i,x_j),\tau(x_j,x_i))$ and $\sigma$, where we set $\sigma=1$ if $x_2$ is not a time endpoint of the triangle and $\sigma=-1$ if $x_2$ is a time endpoint of the triangle. $\sigma$ is called the \emph{sign} of the comparison angle (even though we always have $\tilde{\ma}_{x_2}^{-1}(x_1,x_3)>0$). For a reduction of the number of case distinctions we also define the \emph{signed comparison angle} $\tilde{\ma}_{x_2}^{\mathrm{S},-1}(x_1,x_3)=\sigma\tilde{\ma}_{x_2}^{-1}(x_1,x_3)$.

The $-1$ in the exponent stands for the $K=-1$ in $M_{-1}=\AdS$. We will drop it throughout this document (note that \cite{AngLLS} drops it if $K=0$ instead).
\end{dfn}

\begin{dfn}[Angles]
Let $X$ be a \LpLS and $\alpha,\beta:[0,\varepsilon)\to X$ be two timelike curves (future or past directed or one of each) with $x:=\alpha(0)=\beta(0)$. Then we define the \emph{upper angle} 
\begin{equation*}
\ma_x(\alpha,\beta)=\limsup_{(s,t)\in D, s,t\to 0}\tilde{\ma}_x(\alpha(s),\beta(t))\,,
\end{equation*}
where $D=\{(s,t):s,t>0,\,\alpha(s),\beta(t)\text{ timelike related}\}\cap\{(s,t):\alpha(s),\beta(t),x \\\text{ satisfies size bounds for }K=-1\}$. If the limes superior is a limit and finite, we say \emph{the angle exists} and call it an \emph{angle}.

Note that the sign of the comparison angle is independent of $(s,t)\in D$. We define the \emph{sign} of the (upper) angle $\sigma$ to be that sign, and define the \emph{signed (upper) angle} to be $\ma_x^{\mathrm{S}}(\alpha,\beta)=\sigma \ma_x(\alpha,\beta)$.

If maximisers between any two timelike related points are unique (as e.g.\ in $\AdS'$ (defined in Lem.\ \ref{lem:AdS'})), then we simply write $\ma_p(x,y)$ for the angle at $p$ between the maximisers from $p$ to $x$ and $p$ to $y$.
\end{dfn} 

For giving an alternative definition of timelike curvature bounds in the case $K=-1$ we need:
\begin{dfn}[Regularity]
A \LpLS $X$ is called \emph{regular} if every distance realizer between timelike related points is timelike, i.e., does not contain a null segment.
\end{dfn}
\begin{dfn}[Timelike curvature bounds by monotonicity comparison]
\label{def:monotonicityComp}
Let $X$ be a regular \LpLSn. An open subset $U$ is called a timelike $(\geq K)$-comparison neighbourhood in the sense of \emph{monotonicity comparison} if:
\begin{enumerate}
\item $\tau$ is continuous on $(U\times U) \cap \tau^{-1}([0,D_K))$, and this set is open.
\item For all $x,y \in U$ with $x \ll y$ and $\tau(x,y) < D_K$ there exists a geodesic connecting them which is contained entirely in $U$. 
\item Let $\alpha:[0,a]\to U,\beta:[0,b]\to U$ be distance realizers such that $x:=\alpha(0)=\beta(0)$ and such that $L(\alpha),L(\beta)<D_K$. Define the function $\theta:D\to[0,+\infty)$ by $\theta(s,t):=\tilde{\ma}_x^{K,\mathrm{S}}(\alpha(s),\beta(t))$ ($D\subseteq (0,a]\times(0,b]$ is the set where this is defined). Then $\theta$ is monotonically increasing.
\end{enumerate}
\medskip

We say $X$ has timelike curvature bounded below by $K$ (in the sense of \emph{monotonicity comparison}) if it is covered by timelike $(\geq K)$-comparison neighbourhoods (in the sense of \emph{monotonicity comparison}). 

We say $X$ has global timelike curvature bounded below by $K$ (in the sense of \emph{monotonicity comparison}) if $X$ itself is a $(\geq K)$-comparison neighbourhood (in the sense of \emph{monotonicity comparison}). 
\end{dfn}

\begin{thm}[Timelike curvature: Equivalence of definitions]
\label{thm:equivTriCompAndMonComp}
Let $X$ be a regular \LpLSn. Then $X$ has timelike curvature bounded below (above) by $K=-1$ in the sense of Def.\ \ref{def:triComp} if and only if it has timelike curvature bounded below (above) by $K=-1$ in the sense of monotonicity comparison (see Def.\ \ref{def:monotonicityComp}).
\end{thm}
\begin{proof}
See \cite[Thm.\ 4.12]{AngLLS} or a more complete picture in \cite[Thm.\ 5.1]{equivDefs}.
\end{proof}
For the relation of local and global curvature bounds, note we only need weak additional assumptions to get they are equivalent:
\begin{lem}[Equivalence of local and global lower curvature bounds]
Let $X$ be a connected, globally hyperbolic, regular Lorentzian length space with a time function $T$ and curvature bounded below by $K \in\R$ in the sense of triangle comparison such that any timelike related points are connected by a distance realizer. Then it has global curvature bounded below by $K$ in the sense of triangle comparison.
\end{lem}
\begin{proof}
This follows from \cite[Thm.\ 3.6]{toponogovGlob} and \cite[Thm.\ 5.1 and Prop.\ 4.18]{equivDefs}.
\end{proof}

For technical reasons, we need to assume the timelike geodesic prolongation property. This is unfortunate, as it is a quite strong assumption, among other things it is not satisfied in manifolds with boundary.

\begin{dfn}
Let $X$ be a \LpLSn. It satisfies \emph{timelike geodesic prolongation} if each timelike distance realizer $\gamma:[a,b]\to X$ can be extended as a geodesic to an open domain, i.e.\ there is a $\varepsilon>0$, an extension $\gamma':(a-\varepsilon,b+\varepsilon)\to X$ of $\gamma$ such that both $\gamma'|_{(a-\varepsilon,a+\varepsilon)}$ and $\gamma'|_{(b-\varepsilon,b+\varepsilon)}$ are timelike distance realizers.
\end{dfn}

This plays a technical role in our proof of the Bonnet-Myers rigidity theorem mainly because of the following result.

\begin{pop}[Continuity of angles in spaces with timelike curvature bounded below]
\label{pop:continuityOfAngles}
Let $X$ be a strongly causal, localisable, timelike geodesically connected and locally causally closed \LpLS with timelike curvature bounded below and which satisfies timelike geodesic prolongation. Let $\alpha_n,\alpha,\beta_n,\beta$ be future or past directed timelike geodesics all starting at $\alpha_n(0)=\alpha(0)=\beta_n(0)=\beta(0)=:x$ and with $\alpha_n\to \alpha$ and $\beta_n\to \beta$ pointwise (in particular, $\alpha_n$ is future directed if and only if $\alpha$ is, and similarly for $\beta_n$ and $\beta$). Then
\begin{equation*}
\ma_x(\alpha,\beta)=\lim_n\ma_x(\alpha_n,\beta_n)
\end{equation*}
\end{pop}
\begin{proof}
See \cite[Prop.\ 4.14]{AngLLS}.
\end{proof}

The following two results seem very similar to \cite[Prop.\ 2.42, Prop.\ 2.43]{triSplitting}, except that here the comparison triangles are created in \AdS instead of Minkowski space. For the proof, we refer the reader to the proof of \cite[Prop.\ 2.42, Prop.\ 2.43]{triSplitting} which is for $K=0$, but easily generalizes to general $K$ when additionally assuming size bounds for $K$.

\begin{pop}[Alexandrov lemma: across version]
\label{lem:alexlemAcross}
Let $X$ be a \LpLSn.
Let $\Delta:=\Delta(x, y, z)$ be a timelike triangle satisfying size bounds for $K=-1$ (in particular the distance realisers between the endpoints exist). Let $p$ be a point on the side $xz$ with $p\ll y$, such that the distance realiser between $p$ and $y$ exists. 
Then we can consider the smaller triangles $\Delta_1:=\Delta(x,p,y)$ and $\Delta_2:=\Delta(p,y,z)$. We construct a comparison situation consisting of a comparison triangle $\bar{\Delta}_1$ for $\Delta_1$ and $\bar{\Delta}_2$ for $\Delta_2$, with $\bar{x}$ and $\bar{z}$ on different sides of the line through $\bar{p}\bar{y}$ and a comparison triangle $\tilde{\Delta}$ for $\Delta$ with a comparison point $\tilde{p}$ for $p$ on the side $xz$. This contains the subtriangles $\tilde{\Delta}_1:=\Delta(\tilde{x},\tilde{y},\tilde{p})$ and $\tilde{\Delta}_2:=\Delta(\tilde{p},\tilde{y},\tilde{z})$, see Figure \ref{lem:alexlemAcross:figConcave}.

\begin{figure}
\begin{center}
\begin{tikzpicture}
\draw (-0.5693860319981044,1.9834819014638239)-- (0,0);
\draw (2.8896135929029714,2.2006721639230697)-- (4,0);
\draw (0,0)-- (-0.015287989182225509,1.3000898902049949);
\draw (-0.015287989182225509,1.3000898902049949)-- (-0.5693860319981044,1.9834819014638239);
\draw (-0.5693860319981044,1.9834819014638239)-- (1.0360567397613818,4.954583800321347);
\draw (1.0360567397613818,4.954583800321347)-- (-0.015287989182225509,1.3000898902049949);
\draw (2.8896135929029714,2.2006721639230697)-- (3.647225049614162,4.812946100427443);
\draw (3.647225049614162,4.812946100427443)-- (4,0);
\draw [dashed] (2.8896135929029714,2.2006721639230697)-- (3.904456784270502,1.303506235532433);
\begin{scriptsize}
\coordinate [circle, fill=black, inner sep=0.7pt, label=270: {$\bar{x}$}] (A1) at (0,0);
\coordinate [circle, fill=black, inner sep=0.7pt, label=0: {$\bar{p}$}] (A1) at (-0.015287989182225509,1.3000898902049949);
\coordinate [circle, fill=black, inner sep=0.7pt, label=180: {$\bar{y}$}] (A1) at (-0.5693860319981044,1.9834819014638239);
\coordinate [circle, fill=black, inner sep=0.7pt, label=90: {$\bar{z}$}] (A1) at (1.0360567397613818,4.954583800321347);
\coordinate [circle, fill=black, inner sep=0.7pt, label=270: {$\tilde{x}$}] (A1) at (4,0);
\coordinate [circle, fill=black, inner sep=0.7pt, label=180: {$\tilde{y}$}] (A1) at (2.8896135929029714,2.2006721639230697);
\coordinate [circle, fill=black, inner sep=0.7pt, label=90: {$\tilde{z}$}] (A1) at (3.647225049614162,4.812946100427443);
\coordinate [circle, fill=black, inner sep=0.7pt, label=0: {$\tilde{p}$}] (A1) at (3.904456784270502,1.303506235532433);
\end{scriptsize}
\end{tikzpicture}
\end{center}
\caption{A concave situation in the across version.}
\label{lem:alexlemAcross:figConcave}
\end{figure}
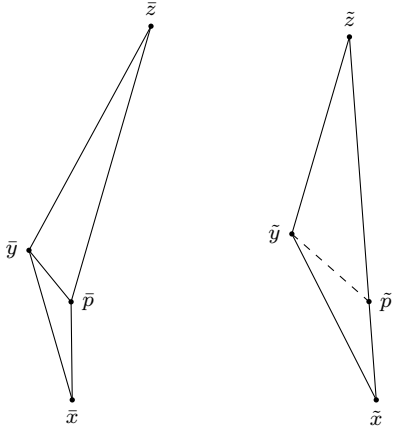

Then the situation $\bar{\Delta}_1$, $\bar{\Delta}_2$ is convex (concave) at $p$ (i.e.\ $\tilde{\ma}_p(x,y)=\ma_{\bar{p}}(\bar{x},\bar{y})\geq\ma_{\bar{p}}(\bar{y},\bar{z})=\tilde{\ma}_p(y,z)$ (or $\leq$)) if and only if $\tau(p,y)\leq\bar{\tau}(\bar{p},\bar{y})$ (or $\geq$).
If this is the case, we have that 
\begin{itemize}
\item each angle in the triangle $\bar{\Delta}_1$ is $\geq$ (or $\leq$) than the corresponding angle in the triangle $\tilde{\Delta}_1$,
\item each angle in the triangle $\bar{\Delta}_2$ is $\geq$ (or $\leq$) than the corresponding angle in the triangle $\tilde{\Delta}_2$.
\end{itemize}
In any case, we have that
\begin{itemize}
\item $\ma_{\bar{y}}(\bar{x},\bar{z}) \geq \ma_{\tilde{x}}(\tilde{x},\tilde{z})=\tilde{\ma}_y(x,z)$.
\end{itemize}
The same is true if $p$ is a point on the side $xz$ such that $y \ll p$.
Note that if $X$ has timelike curvature bounded below (above) by $K=-1$ and $\Delta$ is within a comparison neighbourhood, the condition is satisfied, i.e.\ $\tau(p,y)\leq\bar{\tau}(\tilde{p},\tilde{y})$ (or $\geq$).
\end{pop}

\begin{pop}[Alexandrov lemma: future version]
\label{lem:alexlemFuture}
Let $X$ be a \LpLSn. 
Let $\Delta:=\Delta(x,y,z)$ be a timelike triangle satisfying size bounds for $K=-1$ (in particular the distance realisers between the endpoints exist). Let $p$ be a point on the side $xy$, such that the distance realiser between $p$ and $z$ exists. 
Then we can consider the smaller triangles $\Delta_1:=\Delta(x,p,z)$ and $\Delta_2:=\Delta(p,y,z)$. We construct a comparison situation consisting of a comparison triangle $\bar{\Delta}_1$ for $\Delta_1$ and $\bar{\Delta}_2$ for $\Delta_2$, with $\bar{x}$ and $\bar{y}$ on different sides of the line through $\bar{p}\bar{z}$ and a comparison triangle $\tilde{\Delta}$ for $\Delta$ with a comparison point $\tilde{p}$ for $p$ on the side $xy$. This contains the subtriangles $\tilde{\Delta}_1:=\Delta(\tilde{x},\tilde{p},\tilde{z})$ and $\tilde{\Delta}_2:=\Delta(\tilde{p},\tilde{y},\tilde{z})$, see Figure \ref{lem:alexlemFuture:figConvex}.

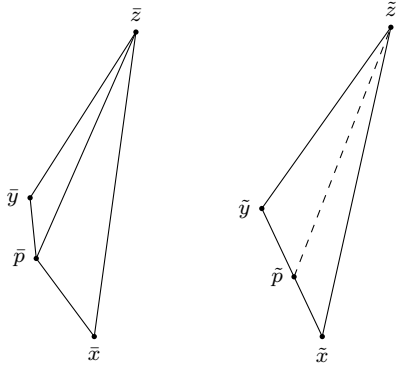
\begin{figure}
\begin{center}
\begin{tikzpicture}
\draw (7,0)-- (6.236686420076234,1.0356870286413933);
\draw (6.236686420076234,1.0356870286413933)-- (7.548286699820673,4.0374024205174575);
\draw (7.548286699820673,4.0374024205174575)-- (7,0);
\draw (6.15561793604921,1.839784099337744)-- (6.236686420076234,1.0356870286413933);
\draw (6.15561793604921,1.839784099337744)-- (7.548286699820673,4.0374024205174575);
\draw (9.204717676576246,1.6977850199451887)-- (10,0);
\draw (9.204717676576246,1.6977850199451887)-- (10.90284745202584,4.1006259914346685);
\draw (10.90284745202584,4.1006259914346685)-- (10,0);
\draw [dashed] (9.628868249068917,0.7922996759744166)-- (10.90284745202584,4.1006259914346685);
\begin{scriptsize}
\coordinate [circle, fill=black, inner sep=0.7pt, label=270: {$\tilde{x}$}] (A1) at (10,0);
\coordinate [circle, fill=black, inner sep=0.7pt, label=180: {$\tilde{p}$}] (A1) at (9.628868249068917,0.7922996759744166);
\coordinate [circle, fill=black, inner sep=0.7pt, label=270: {$\bar{x}$}] (A1) at (7,0);
\coordinate [circle, fill=black, inner sep=0.7pt, label=180: {$\bar{p}$}] (A1) at (6.236686420076234,1.0356870286413933);
\coordinate [circle, fill=black, inner sep=0.7pt, label=90: {$\bar{z}$}] (A1) at (7.548286699820673,4.0374024205174575);
\coordinate [circle, fill=black, inner sep=0.7pt, label=180: {$\bar{y}$}] (A1) at (6.15561793604921,1.839784099337744);
\coordinate [circle, fill=black, inner sep=0.7pt, label=180: {$\tilde{y}$}] (A1) at (9.204717676576246,1.6977850199451887);
\coordinate [circle, fill=black, inner sep=0.7pt, label=90: {$\tilde{z}$}] (A1) at (10.90284745202584,4.1006259914346685);
\end{scriptsize}
\end{tikzpicture}
\end{center}
\caption{A convex situation in the future version.}
\label{lem:alexlemFuture:figConvex}
\end{figure}

Then the situation $\bar{\Delta}_1$,$\bar{\Delta}_2$ is convex (concave) at $p$ (i.e.\
$\ma_{\bar{p}}(\bar{y},\bar{z})\leq\ma_{\bar{p}}(\bar{x},\bar{z})$ (or $\geq$)) if and only if $\tau(p,z)\leq\bar{\tau}(\bar{p},\bar{z})$ (or $\geq$). If this is the case, we have that
\begin{itemize}
\item each angle in the triangle $\bar{\Delta}_1$ is $\geq$ (or $\leq$) than the corresponding angle in the triangle $\tilde{\Delta}_1$,
\item each angle in the triangle $\bar{\Delta}_2$ is $\leq$ (or $\geq$) than the corresponding angle in the triangle, $\tilde{\Delta}_2$.
\end{itemize}
In any case, we have that
\begin{itemize}
\item $\ma_{\bar{z}}(\bar{x},\bar{y})\leq\ma_{\tilde{z}}(\tilde{x},\tilde{y})=\tilde{\ma}_z(x,y)$.
\end{itemize}
Note that if $X$ has timelike curvature bounded below (above) by $K=-1$ and $\Delta$ is within a comparison neighbourhood, the condition is satisfied, i.e.\ $\tau(p,z)\leq\bar{\tau}(\tilde{p},\tilde{z})$ (or $\geq$).
\end{pop}
\section{Lorentzian warped products}\label{sec:warpedproducts}
This section is dedicated to the treatment of Lorentzian warped product spaces. In the context of Lorentzian pre-length spaces, warped products were first introduced in \cite{generalizedCones}. Here, we give a different treatment which also works for non-intrinsic spaces, and is in addition more suited for our needs. First, for technical reasons, we define the following:

\begin{dfn}[Lorentzian warped product comparison space]
Let $I\subseteq\mb{R}$ be an open interval and $f:I\to(0,+\infty)$ be continuous. The \emph{Lorentzian warped product comparison space for $f$} or \emph{with warping function $f$} is the warped product $X=I\times_f \R$ as a spacetime, i.e.\ equipping the manifold $I\times \R$ with the continuous metric $g=-\diff t^2 + (f\circ t)^2 \diff x^2$ and the time orientation given by $\partial_t$. We will use this as a \LpLSn, with the product metric $D$ as metric. We set $t:X\to I$ the first and $x:X\to\mb{R}$ the second projection.
\end{dfn}

\begin{pop}[Properties of Lorentzian warped product comparison spaces]\label{pop:PropsLwpcs}
Let $f:I\to(0,+\infty)$ be continuous. Then the Lorentzian warped product comparison space $X=I\times_f \R$ has the following properties:
\begin{enumerate}
\item \label{pop:PropsLwpcs:descrNullGen} The reparametrized null distance realizers are given as:
\begin{equation*}
\gamma(t)=\left(t,\pm\int_{t_0}^t \frac{1}{f(\lambda)}\diff \lambda\right)
\end{equation*}
\item \label{pop:PropsLwpcs:descrLeq} In particular, $(s,x)\leq (t,y) \iff |x-y|\leq\int_s^t \frac{1}{f(\lambda)}\diff \lambda$ and $s\leq t$, and $(s,x)\ll (t,y)$ iff these inequalities are strict.
\item \label{pop:PropsLwpcs:causPlain} It is causally plain (see \cite[Def.\ 1.16]{chruscielGrant}). In particular it forms a Lorentzian length space. 
\item It is globally hyperbolic, globally causally closed, strongly localizable and strictly intrinsic. $\tau_{I\times_f \R}$ is continuous.
\item \label{pop:PropsLwpcs:regular} If $f$ is Lipschitz, it is even a regular Lorentzian length space.
\item \label{pop:PropsLwpcs:geodesiceqn} If $f$ is $C^1$, we can describe timelike geodesics with the ODE $\gamma=(\alpha,\beta)$, $\alpha''=g(\beta',\beta')(ff')\circ\alpha$, $\beta''=\frac{-2 (f\circ\alpha)'}{f\circ\alpha}$.
\end{enumerate}
\end{pop}
\begin{proof}
(i): Note that these are null curves, and as the space is $2$-dimensional, they are reparametrized null geodesics. As any other causal curve parametrized w.r.t.\ $t$-coordinate has less speed in the $x$-coordinate, it is a distance realizer.

(ii): The description of the causal relation follows immediately (there is only this one causal curve between null related points), and one can easily modify $\gamma$ to be timelike for the strict inequality, see part (iii) for more details.

(iii): We now show that $X$ is causally plain: Let $p_0=(t_0,x_0)\in X$. As $f$ is continuous, there is a neighbourhood $(t_-,t_+)\ni t_0$ such that $f(s)\in (\frac{f(t_0)}{2},2f(t_0))$ for all $s\in(t_-,t_+)$. Using the coordinates $t$ and $\frac{x}{f(t_0)}$, we see that $U:=(t_-,t_+)\times\mb{R}$ is a cylindrical neighbourhood (as defined in \cite[Def.\ 1.8]{chruscielGrant}). One checks that $p_0\ll q_1=(t_1,x_1)$ if and only if $|x_1-x_0|<\int_{t_0}^{t_1} \frac{1}{f(\lambda)}\diff\lambda=:x(t_0,t_1)$. There now is a smooth function $\tilde{f}$ such that $f<\tilde{f}<\frac{x(t_0,t_1)}{|x_1-x_0|}f$. Then $\tilde{\gamma}(t)=(t,x_0\pm\int_{t_0}^t\frac{|x_1-x_0|}{x(t_0,t_1)}\frac{1}{f(\lambda)}\diff\lambda)$ connects $p_0$ with $p_1$ and is timelike with respect to the smooth metric $-\diff t^2+\tilde{f}(t)^2\diff x^2$ which has narrower lightcones than $g$. In particular, one sees that $\check{I}^+(p_0,U)=I(p_0)\cap U$. Now $\partial\check{I}(p,U)=\partial J(p,U)$ follows from $J(p,U)=\overline{I(p,U)}\cap U$.

(iv): Global causal closedness follows from the description of the causal relation (note that $f$ is locally bounded away from $0$, so the integral stays finite). 

It is non-totally imprisoning: the $d$-length of a causal curve $\gamma=(\id,\beta):[a,b]\to X$ is bounded above by $(b-a)\sqrt{1+\frac{1}{\min\{f(t):t\in[a,b]\}^2}}$, in particular we have a uniform bound on sets of the form $[a,b]\times\mb{R}$.

Causal diamonds are compact: A causal diamond $J(p_0,p_1)$ is closed by global causal closedness, and it is compact as $\min\{f(t):t\in[t_0,t_1]\}>0$, so it is contained in the compact $[t_0,t_1]\times[x_0-\frac{1}{\min\{f(t):t\in[t_0,t_1]\}},x_0+\frac{1}{\min\{f(t):t\in[t_0,t_1]\}}]$ (as $\min\{f(t):t\in[t_0,t_1]\}>0$). In particular, we have global hyperbolicity, continuity of $\tau$ and strict intrinsicness.

It is strongly localizable: every timelike diamond $U$ is strictly intrinsic, has continuous $\tau$, there are no $\ll$-isolated points in $U$ and it is a $d$-compatible neighbourhood.

(v): We have a globally hyperbolic Lorentzian manifold with Lipschitz metric, so can apply \cite[Prop.\ 1.2]{LorentzLipschitz}.

(vi): If $f$ is $C^1$, the geodesic equation follows from \cite[Prop.\ 7.38]{ONeill}.
\end{proof}

\begin{dfn}[Lorentzian warped product]
\label{def:lorWProduct}
Let $I\subseteq\mb{R}$ be an interval and $f:I\to(0,+\infty)$ be continuous. Let $(Y,d)$ be a metric space, and let $g:I\to\mb{R}$ be a homeomorphism. Define the space $X:= I\times Y$. Equip it with the product metric 
\begin{equation*}
D: X \times X \to \R, \ D((s,x),(t,y)):=\sqrt{|g(t)-g(s)|^2+d(x,y)^2}.
\end{equation*}
Define the timelike relation as $(s,x) \ll (t,y)$ if and only if in the Lorentzian warped product comparison space $I\times_f \R$, we have that $(s,0) \ll_{I\times_f \R} (t,d(x,y))$. Similarly, they are causally related if these points in $I\times_f \R$ are causally related, and set
\[
\tau((s,x),(t,y))=\tau_{I\times_f \R}((s,0),(t,d(x,y)))
\]

Then $(X,D,\ll,\leq,\tau)$ is called the Lorentzian warped product of $Y$ with $\R$ with warping function $f$. It will be denoted by $\R\times_f Y$. We set $t:X\to I$ the first and $x:X\to Y$ the second projection.
\end{dfn}

\begin{pop}[Properties of Lorentzian warped products]
\label{pop:PropsLwp}
Let $(Y,d)$ be a metric space and $f:I\to(0,+\infty)$ be continuous. Set $X:=I\times_f Y$ to be the Lorentzian warped product. Then:
\begin{enumerate}
\item \label{pop:PropsLwp:nonTotImpr} $X$ is a globally causally closed and non-totally imprisoning \LpLS with continuous $\tau$.
\item $X$ is strongly causal (the timelike diamonds even form a basis for the topology).
\item \label{pop:PropsLwp:Lip} All causal curves in $t$-parametrization on some $[s,t]$ are uniformly Lipschitz, in particular $X$ is $d$-compatible and any causal curve is locally Lipschitz.
\item \label{pop:PropsLwp:geodesics} A causal curve $\gamma=(\alpha,\beta):[0,L]\to X$ is a distance realizer if and only if $\beta$ is a distance realizer and $\tilde{\gamma}(\lambda)=(\alpha(\lambda),d(\beta(0),\beta(\lambda)))$, $\tilde{\gamma}:[0,L]\to I\times_f \mb{R}$ is a distance realizer into the warped product comparison space.
\item In particular, if $Y$ is strictly intrinsic, $X$ is strictly intrinsic. 
\item If $Y$ is strictly intrinsic, $X$ is strongly localizable. If this is the case and $f$ is Lipschitz, $X$ is regularly localizable.
\item \label{pop:PropsLwp:globHyp} If $Y$ is proper $X$ is globally hyperbolic, and the converse holds if $\int_I\frac{1}{f(\lambda)}\diff\lambda=+\infty$.
\item If $Y$ is proper, $X$ is proper as a metric space.
\end{enumerate}
\end{pop}
\begin{proof}
(i) $\tau$ is continuous since $\tau_{I\times_f \R}$ is. Similarly, $\leq$ is closed on $X \times X$. The other properties of \LpLS follow analogously, except for the transitivity of the relations and the reverse triangle inequality. The transitivity follows from the reverse triangle inequality. The reverse triangle inequality can either be seen directly from a $2+1$-dimensional Lorentzian warped product comparison space, or as follows:

For the reverse triangle inequality, let $p=(r,x)\leq q=(s,y)\leq r=(t,z)$. Then we have the triangle inequality of $Y$: $d(x,z)\leq d(x,y)+d(y,z)$. We find the time separations between the points using points in the Lorentzian warped product comparison space: $\tilde{p}=(r,0),\,\tilde{q}=(s,d(x,y)),\,\tilde{r}=(t,d(x,y)+d(y,z))$ and $\hat{r}=(t,d(x,z))$. Then we have 
\begin{align*}
\tau(p,q)&=\tau_{I\times_f \R}(\tilde{p},\tilde{q})\\
\tau(q,r)&=\tau_{I\times_f \R}(\tilde{q},\tilde{r})\\
\tau(p,r)&=\tau_{I\times_f \R}(\tilde{p},\hat{r})
\end{align*}
By the reverse triangle inequality we have $\tau_{I\times_f \R}(\tilde{p},\tilde{r})\geq\tau_{I\times_f \R}(\tilde{p},\tilde{q})+ \tau_{I\times_f \R}(\tilde{q},\tilde{r})$, and as the $x$-coordinate difference is smaller between $\tilde{p}\hat{r}$ than between $\tilde{p}\tilde{r}$ (with equal $t$-coordinates), we can take a distance realizer $\gamma=(\alpha,\beta)$ from $\tilde{p}$ to $\tilde{r}$ and convert it to a longer causal curve $\gamma_2=(\alpha,\frac{d(x,z)}{d(x,y)+d(y,z)}\beta)$ from $\tilde{p}$ to $\hat{r}$.

For non-total imprisonment, let $\tilde{K}$ be compact, so it is contained in some $K=[a,c]\times Y$. Set $\inf f|_{[a,c]}=:f_->0$, then for any two points $(s,y)\leq(t,z)$ we have $d(y,z)\leq \frac{t-s}{f_-}$, so $D((s,y),(t,z))=\sqrt{(t-s)^2+d(y,z)^2}\leq\sqrt{1+\frac{1}{f_-}}(t-s)$. In particular, for any causal curve $\gamma$ starting at $(s,y)$ and ending at $(t,z)$ we have $L_D(\gamma)\leq \sqrt{1+\frac{1}{f_-}}(t-s)$, and if $\gamma$ stays in $K$ this is less than $\sqrt{1+\frac{1}{f_-}}(c-a)$ and $X$ is non-totally imprisoning. 

(ii) To show that the diamonds form a basis, let $(b,y)\in O=(a,c)\times B_R(x)$ (for some $a<c \in I$, $R>0$, $x,y\in Y$) with $[a,c]\subseteq I$. Then we have that $\inf f|_{[a,c]}=:f_->0$. Take $0<\varepsilon=\min(b-a,c-b,(R-d(x,y))f_-/2)$, set $p:=(b-\varepsilon,y),q:=(b+\varepsilon,y) \in \bar{O}$. We claim $(b,y)\in I(p,q)\subseteq O$: As $\int_a^c\frac{1}{f(\lambda)}\diff\lambda\leq\frac{(c-a)}{f_-}=\frac{2\varepsilon}{f_-}\leq R-d(x,y)$ and this is the maximum distance a null related point can be away in the $Y$ coordinate, we get that $I(p,q)\subseteq O$.

(iii) Any causal curve in $t$-parametrization $\gamma(t)=(t,\beta(t))$ is locally Lipschitz: Note that by \ref{pop:PropsLwpcs}.\ref{pop:PropsLwpcs:descrLeq}, we get that $d(\beta(s),\beta(t))\leq\int_s^t\frac{1}{f(\lambda)}\diff\lambda$, thus we get that $D(\gamma(s),\gamma(t))\leq\sqrt{(t-s)^2+d(\beta(s),\beta(t))^2}\leq(t-s)\sqrt{1+\frac{1}{\min_{[s,t]}f}}$, and the Lipschitz constant only depends on $s,t$.

(iv) Now for the equivalent condition for distance realizers and their existence. First, let $\gamma=(\alpha,\beta):[0,L]\to X$ be a distance realizer. Set $p=(r,\beta(r)),\,q=(s,\beta(s)),\,r=(t,\beta(t))$. We are now in the same situation as in the proof for the reverse triangle inequality, and note that in he proof of the reverse triangle inequality, equality can only be achieved if $d(\beta(r),\beta(t))=d(\beta(r),\beta(s))+d(\beta(s),\beta(t))$, so $\beta$ has to be a distance realizer. \\
Now note that by definition of $\tau$, the homeomorphism from the two-dimensional subset $\{(t,\beta(\lambda)):t,\lambda\}\subseteq X$ to $I\times [0,L(\beta)]\subseteq I\times_f\mb{R}$ given by $(t,x)\mapsto (t,d(x,\beta(0)))$ is $\tau$-preserving, so they are in bijection as required. 

(v) As $I\times_f\mb{R}$ is globally hyperbolic, there is a distance realizer $\tilde{\gamma}$ between any causally related points (\cite[Prop.\ 6.4]{globHypC0Spt}), so the existence is established if $\beta$ exists.

(vi) For the strong localizability, we claim that every timelike diamond is a localizable neighbourhood, and the proof is as for the Lorentzian warped product comparison space in Prop.\ \ref{pop:PropsLwpcs}.\ref{pop:PropsLwpcs:causPlain}. $X$ being regular is also done in Prop.\ \ref{pop:PropsLwpcs}.\ref{pop:PropsLwpcs:regular}.

(vii) First let $Y$ be proper, for global hyperbolicity we still need to check the compactness of causal diamonds. This works as for the Lorentzian warped product comparison space, but instead of the closed interval in space, we now need a closed ball (which is compact by properness of $Y$).

Now let $X$ be globally hyperbolic and $\int_I\frac{1}{f(\lambda)}\diff\lambda=+\infty$. We need that $Y$ is proper. So let $x\in Y$ and $R>0$. Take $t_0<t_1<t_2\in I$ be such that $\int_{t_0}^{t_1}\frac{1}{f(\lambda)}\diff\lambda=\int_{t_1}^{t_2}\frac{1}{f(\lambda)}\diff\lambda=R$, then $x(J((t_0,x),(t_2,x)))=\bar{B}_R(x)\subseteq Y$: First note that $(t_0,x)\leq (t,y)$ if and only if $d(x,y)\leq\int_{t_0}^t\frac{1}{f(\lambda)}\diff\lambda$ by Prop.\ \ref{pop:PropsLwpcs}.\ref{pop:PropsLwpcs:descrLeq}, similarly for $(t,y)\leq (t_2,x)$. Thus one gets that $x(J((t_0,x),(t_2,x)))\subseteq \bar{B}_R(x)$, and for $t=t_1$ we have that $x(J((t_0,x),(t_2,x))\cap\{(t_1,y):y\in Y\})= \bar{B}_R(x)$.

(viii) As $g$ maps $I$ into $\mb{R}$ bijectively, $X$ seen as a metric space is isometric to the metric product $\mb{R}\times Y$, in particular it will be proper if $Y$ is.
\end{proof}

We want to compare this to the already known notion of Lorentzian warped products from \cite[Def.\ 3.2,3.4,3.17,3.18]{generalizedCones}
\begin{dfn}[Lorentzian warped products as in \cite{generalizedCones}]\label{def:wProdGenCones}
Let $I\subseteq\mb{R}$ be an open interval and $f:I\to(0,+\infty)$ continuous. Furthermore, let $Y$ be a metric space. The Lorentzian warped product $I\times_f Y$ in the sense of \cite{generalizedCones} is constructed as follows: as a set it is  $I\times Y$, together with the product metric $D((s,x),(t,y))=|s-t|+d(x,y)$. 

For an absolutely continuous curve $\gamma=(\alpha,\beta):J\to I\times_f Y$, we say it is \emph{future directed causal as in \cite{generalizedCones}} if $\alpha$ is strictly increasing and $-\dot\alpha^2+(f\circ\alpha)^2v_\beta^2\leq0$, where $v_\beta$ denotes the metric derivative of $\beta$ (existing almost everywhere). It is \emph{future directed timelike} if additionally $-\dot\alpha^2+(f\circ\alpha)^2v_\beta^2<0$.

For a future directed causal curve $\gamma:[a,b]\to I\times_f Y$, we define its length as in \cite{generalizedCones} as
\[
\tilde L(\gamma)=\int_a^b\sqrt{\dot\alpha^2-(f\circ\alpha)^2v_\beta^2}\,.
\]

The \emph{time separation function as in \cite{generalizedCones}} is defined as:
\[
\tilde\tau(p,q)=\sup\{\tilde L(\gamma): \gamma\text{ future directed causal curve from }p\text{ to }q\}
\]
and $\tilde\tau(p,q)=0$ if this set is empty.

The causal and timelike relation are defined as $p\mathrel{\tilde\leq}q$ if there exists a future directed causal curve as in \cite{generalizedCones} from $p$ to $q$, similarly $p\mathrel{\tilde\ll}q$ for timelike.
\end{dfn}

\begin{pop}[The warped product definition agrees with the known notion]
Let $(Y,d)$ be a metric space and $f:I\to (0,+\infty)$ be continuous. Then we can view $I\times_fY$ as a warped product in two ways: As in Def.\ \ref{def:lorWProduct}, or as in \cite{generalizedCones} (see \ref{def:wProdGenCones}). 

We have that $\tau\geq\tilde\tau$ (in particular, $\tilde\ll$ is contained in $\ll$) and $\tilde\leq$ is contained in $\leq$.

If $Y$ is intrinsic, the time separation and the timelike relation agree, i.e.\ $\tau=\tilde\tau$ and ${\ll}=\tilde\ll$.

Moreover, if $Y$ is strictly intrinsic the causal relation agrees, i.e.\ ${\leq}=\tilde\leq$. \footnote{Note that for both intrinsic and strictly intrinsic, it is enough to require it up to $\int_I\frac1f$ (i.e.\ all $x,y\in Y$ with distance $d(x,y)<\int_I\frac1f$ are connected by a $d$-distance realizer or a sequence whose length is converging to the distance): If $d(x,y)\geq\int_I\frac1f$ we have that $(s,x)\not\leq (t,y)$ for any $s,t$ by \ref{pop:PropsLwpcs}.\ref{pop:PropsLwpcs:descrLeq}.}
\end{pop}
\begin{proof}
If $X$ need not be intrinsic, for any \cite{generalizedCones}-causal curve $\gamma_0=(\alpha_0,\beta_0)$ from $(s,x)$ to $(t,y)$, we know $d(x,y)\leq \tilde L(\beta_0)$. Now set $\bar\beta_0(t)=\tilde L(\beta_0,0,t)$, then $\gamma_0$ has the same length as $\bar{\gamma}_0=(\alpha_0,\bar\beta_0)$ in the Lorentzian warped product comparison space for $f$. Setting $\bar\beta_1(t)=\frac{d(x,y)}{\tilde L(\beta_0)} \tilde L(\beta_0,0,t)$, we have that $\bar{\gamma}_1=(\alpha_0,\bar\beta_1)$ is causal and longer (strictly if $d(x,y)\leq \tilde L(\beta_0)$ was strict). In particular, $\tau\geq\tilde{\tau}$ and $\tilde{\leq}$ is contained in ${\leq}$.

We now prove two properties of the definition in \cite{generalizedCones}: First, the $\tilde\tau$-separation $\tilde\tau((s,x),(t,y))$ is given by the limit of lengths of a sequence of \cite{generalizedCones}-causal curves $\gamma_n:[0,b]\to X$, $\gamma_n(t)=(\alpha_n(t),\beta_n(t))$, where $\beta_n:[0,b]\to Y$ is independent of $n$ and a minimizer in $Y$, or, if that does not exist, a minimizing sequence (dependent on $n$).

For this, note that if $\beta:[0,b]\to Y$ is not the shortest curve between $x$ and $y$, we can construct a \cite{generalizedCones}-timelike curve from $(s,x)$ to $(t,y)$ which is strictly longer than $\gamma$: Let $\tilde L(t):=\tilde L(\beta|_{[0,t]})$. Let $\bar{\beta}$ be a shorter curve from $x$ to $y$ with parametrization such that $\tilde L(\bar{\beta}|_{[0,t]})=\frac{\tilde L(\bar\beta)}{\tilde L(\beta)}\tilde L(t)$. As $\frac{\tilde L(\bar\beta)}{\tilde L(\beta)}<1$, we have that also $\bar{\gamma}=(\alpha,\bar{\beta})$ is \cite{generalizedCones}-timelike from $(s,x)$ to $(t,y)$, and that $\bar{\gamma}$ is strictly longer. 

For the second property, by definition two points $(s,x),(t,y)$ are \cite{generalizedCones}-null related if and only if there is a \cite{generalizedCones}-null curve $\gamma:[0,b]\to X$, $\gamma(t)=(\alpha(t),\beta(t))$. By the above argument we can conclude $\beta:[0,b]\to Y$ has to minimize the length from $x$ to $y$. 

Now assume that $Y$ is strictly intrinsic, and take $\beta$ from $x$ to $y$ to be a $d$-distance realizer, and we want to prove $\tau$ and $\leq$ agree: by the above, we can restrict ourselves to causal curves of the form $(\alpha,\beta)$ for some $\alpha$. We define a curve in the Lorentzian warped product comparison space for $f$: let $\bar\beta(t)=d(x,\beta(t))$ be the curve $\bar{\beta}:[0,b]\to\mb{R}$, set $\bar{\gamma}(t)=(\alpha(t),\bar\beta(t))$. It is causal in the Lorentzian warped product comparison space and has the same length as $\gamma$ because $\bar\beta$ has the same metric speed as $\beta$ and the warped product construction agrees with the continuous Lorentzian manifold notion. Now we have $\tau=\tilde{\tau}$: both are the supremum of such $\bar{\gamma}$ in the Lorentzian warped product comparison space (and this supremum is achieved). The same argument also proves ${\leq}={\tilde{\leq}}$. 

If we have that $Y$ is only intrinsic and we have $L:=\tau((s,x),(t,y))>0$, we get $\gamma=(\alpha,\beta)$ in the Lorentzian warped product comparison space with $L(\gamma)=L$, and we can assume $\alpha(t)=t$ (so $\gamma$ is parametrized by $t$-coordinate and $\gamma$ is Lipschitz). We have that 
\begin{equation*}
0<L(\gamma)=\int\sqrt{(\alpha')^2-(f\circ\alpha)^2|\beta'|^2}
\end{equation*}
In particular, for all $\varepsilon>0$ there is a small enough $\delta>0$ and an almost distance realizer $\tilde\beta:x\leadsto y$ in $Y$ with $L_d(\tilde\beta)<(1+\delta)d(x,y)$ such that there is a parametrization of $\tilde\beta$ satisfying
\begin{equation*}
\frac{\varepsilon}{(1-\varepsilon)(f\circ\alpha)^2}(\alpha')^2\geq (|\tilde\beta'|)^2-(|\beta'|)^2
\end{equation*}
so
\begin{equation*}
(f\circ\alpha)^2(1-\varepsilon)(|\beta'|)^2)-(f\circ\alpha)^2(1-\varepsilon)(|\tilde\beta'|)^2+ \varepsilon(\alpha')^2\geq0
\end{equation*}
We add this inside the square root of the following:
\begin{align*}
\sqrt{(1-\varepsilon)}L(\gamma)&=\int\sqrt{(1-\varepsilon)\left((\alpha')^2-(f\circ\alpha)^2|\beta'|^2\right)}\\
&=\int\sqrt{(1-\varepsilon)(\alpha')^2-(f\circ\alpha)^2(1-\varepsilon)(|\beta'|)^2}\\
&\leq \int\sqrt{(\alpha')^2-(f\circ\alpha)^2(1-\varepsilon)(|\tilde\beta'|)^2}\leq L_{\tau}(\tilde\gamma)\\
\end{align*}
for the curve $\tilde\gamma=(\alpha,\tilde\beta)$, i.e.\ we have $\tilde L(\tilde\gamma)\geq \sqrt{1-\varepsilon}L((\alpha,\beta))$. In particular, the definitions of $\ll$ and $\tau$ agree. 
\end{proof}
\begin{dfn}
Let $I\times_f Y$ be a warped product with $0\in I$. \emph{Conformal time} is given by $\eta(t)=\int_0^t \frac{1}{f(s)}\diff s$. 
\end{dfn}
\begin{lem}[Conformal time]
We have that for $p,q\in Y$ and $s\in\eta^{-1}(I)$, $(\eta^{-1}(s),p)\leq (\eta^{-1}(s+d(p,q)),q)$ are null related.
\end{lem}
\begin{proof}
We have to verify that the points $\tilde{p}=(\eta^{-1}(s),0)$ and $\tilde{q}=(\eta^{-1}(s+d(p,q)),d(p,q))$ are null related. For this, note that the curve $t\mapsto(\eta^{-1}(s+t),t)$ is a reparametrization of $\gamma(t)=(s+t,\int_s^{s+t} \frac{1}{f(\lambda)}\diff \lambda)$ which is a null distance realizer by \ref{pop:PropsLwpcs}.\ref{pop:PropsLwpcs:descrNullGen}, thus all points along it are null related.
\end{proof}

\subsection{A warped product in Anti-deSitter}\label{subsec:warpedproducts:AdSasWarpedProd}
First, we collect the following properties of a certain subset of anti-deSitter space:
\begin{lem}[A warped product inside Anti-deSitter]\label{lem:AdS'}
Take $\gamma:[-\pi,0]\to\AdS$ to be a lift of the curve $t\mapsto (\cos(t),\sin(t),0)$ and let $p_-$ and $p_+$ be the start and endpoint of $\gamma$. Then they are lifts of $(-1,0,0),(1,0,0)\in\mb{R}^{1,2}$ in such a way that there is a future-directed geodesic from $p_-$ to $p_+$, but it does not ``wrap around'' the space $\AdS/\mb{Z}$ (i.e.\ they are as close as possible). Set $X=I(p_-,p_+)$. Then $X$ is globally hyperbolic and maximal as such (i.e.\ there is no $X\subsetneq \tilde X \subseteq\AdS$ which is connected and globally hyperbolic). Furthermore, $X$ is isomorphic to the warped product (even warped product comparison space) $\AdS':=(-\frac{\pi}{2},\frac{\pi}{2})\times_{\cos}\mb{R}$. Timelike triangles (in any space) satisfying the size bounds for $K=-1$ are realizable in $\AdS'$, with longest side having constant second coordinate.
\end{lem}
\begin{figure}[htb]
\includegraphics[scale=0.3]{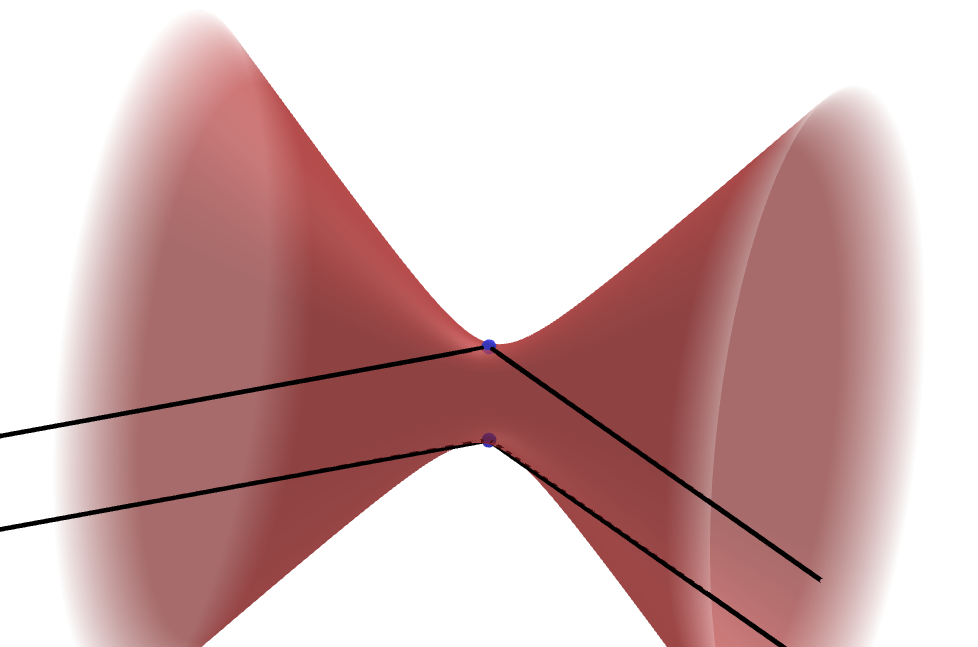}
\caption{The region between the lines is the maximal globally hyperbolic subset of \AdS. Note the two rays going to the left are parallel, similarly for those to the right.}
\end{figure}
\begin{proof}
In $\AdS$, the $c+\frac{\pi}{2}$-level set of $\tau(p_-,\cdot)$ is given by $\AdS\cap\{(s_1,s_2,z):s_1=\sin(c)\}$. All timelike geodesics through $p_-,p_+$ are given by \linebreak$t\mapsto(\sin(t),\cos(t)\cosh(x),\cos(t)\sinh(x))$ (the $x$ parametrizes the angles at $p$), and they are distance realizing, thus pass at right angles through the level sets. 

So we define $f:\AdS' \to X$ by setting 
\begin{equation*}
f(t,x)=(\sin(t),\cos(t)\cosh(x),\cos(t)\sinh(x))\,.
\end{equation*}
This is an isometry: The vertical unit speed distance realizers $t\mapsto (t,x_0)$ map to the unit speed distance realizers $t\mapsto(\sin(t),\cos(t)\cosh(x_0),\cos(t)\sinh(x_0))$ (which cover all $\AdS'$ resp.\ $X$), and the horizontal spacelike geodesic $x\mapsto (t_0,x)$ has speed $\cos(t_0)$ and maps to $x\mapsto(\sin(t_0),\cos(t_0)\cosh(x),\cos(t_0)\sinh(x))$ which is also a spacelike geodesic of speed $\cos(t_0)$ (these are parametrisations of the level sets of $\tau(p_-,\cdot)$, thus cover all $\AdS'$ resp.\ $X$), and the geodesics of the first family intersect the ones in the second family at right angles as described before.

That $X$ is globally hyperbolic can be seen in the warped product picture. 

For realizing timelike triangles satisfying the size bonds for $K=-1$, let $a+b\leq c<\pi$. Choose a realization with a timelike triangle $\Delta(p,q,r)$ in \AdSn. By applying a suitable isometry, we can w.l.o.g.\ assume that $p=(\cos(-c/2),\sin(-c/2),0),\,r=(\cos(c/2),\sin(c/2),0) \in X$ (these have the right $\tau$-distance). Then $q\in X$ by causal convexity of $X$, so $\Delta(p,q,r)$ is realized in $X$, or equivalently in $\AdS'$ where $p,r$ have the same $x$-coordinate ("vertical").

For the inextendibility of $X$, we indirectly assume an extension $X'$ of $X$ was still globally hyperbolic. Connect $p_-$ to $p_+$ by a timelike distance realizer $\gamma$. First, if there is a point $p\in I^+(\gamma)\setminus I^-(\gamma)$, we find a $t$ such that $\gamma(t)\ll p$, and a timelike distance realizer $\alpha_t$ connecting them. Inside a localizing neighbourhood of $p$, we have a point $p\ll p_+$ and set $\varepsilon=2\tau(p,p_+)>0$. $\alpha_t$ is initially contained in $I(\gamma)$, so we find a point $\alpha_t(s_t)\in\partial I(\gamma)$, i.e.\ in $\partial I^-(\gamma)$. 

By Lem.\ \ref{lem:causalBdryAtPi} we get that for all $\varepsilon>0$ there is $\delta>0$ small enough such that $\tau(\gamma(r),\alpha_t(s_t-\delta))\to\pi-\varepsilon$ as $r\to-\frac{\pi}{2}$. In particular, $\tau(\gamma(r),p_+)>\pi$. By the synthetic Bonnet-Myers theorem \ref{thm: lor meyers}, we know that this forces $\tau(\gamma(r),p_+)=+\infty$, contradicting finiteness of $\tau$ (see \cite[Thm.\ 3.28]{saemLorLen}).

The situation where there is a point $p\in I^-(\gamma)\setminus I^+(\gamma)$ is analogous. 

Finally, assume that all points in $X \setminus I(\gamma)$ are in neither $I^+(\gamma)$ nor in $I^-(\gamma)$. Take any such point $p$, then by strong causality, we have $p_-\ll p\ll p_+$. Then also $p_-,p_+$ are neither in $I^+(\gamma)$ nor in $I^-(\gamma)$. Thus, $I(p_-,p_+)\subseteq X \setminus I(\gamma)$ is an open neighbourhood of $p$. This works for any $p$, so both $I(\gamma)$ and $X \setminus I(\gamma)$ are open, contradicting connectedness of $X$.
\end{proof}
\begin{rem}
Note that this subset is invariant under some, but not all isometries of \AdSn: all isometries of $\AdS$ are described as Lorentz transformations of the ambient $\mb{R}^{1,2}$. Isometries fixing $X$ are those fixing $p_-$ (and thus automatically $p_+$). Those have dimension $1$ and codimension $2$, and are Lorentz boosts in the $s_2,z$ directions. In the warped product setting, they are described as translations in the $x$-direction. 
\end{rem}
For \AdS, we have an explicit solution of the geodesic equation \ref{pop:PropsLwpcs}.\ref{pop:PropsLwpcs:geodesiceqn}:
\begin{lem}\label{lem:geodInAdS'}
The timelike geodesics in $\AdS'$, parametrized by $\tau$-arclength and such that at parameter $\lambda=0$ they cross $t=0$, are given by:
\begin{equation*}
\alpha(\lambda)= \left(\arcsin(\sin(\lambda)\cosh(\omega)),\sinh^{-1}\left(\frac{\sin(\lambda)\sinh(\omega)}{\sqrt{1-\sin(\lambda)^2\cosh(\omega)^2}}\right) + c\right)
\end{equation*}
for $\omega,c\in\mb{R}$, and these are distance realizers on their domain. $c$ is the $x$-coordinate at $t=\lambda=0$, $\omega$ that of the angle to the vertical at $t=\lambda=0$.

For $\omega=0$, this is a vertical line $\lambda\mapsto(\lambda,c)$. For $\omega>0$, set $\lambda_{\pm}=\pm\arcsin(\frac{1}{\cosh(\omega)})$, then $\lim_{\lambda\to\lambda_-}\alpha(\lambda)=(-\frac{\pi}{2},-\infty)$ and $\lim_{\lambda\to\lambda_+}\alpha(\lambda)=(+\frac{\pi}{2},+\infty)$.
For $\omega<0$, the signs of the $x$-coordinates of the limits flip. In particular, the $x$-coordinate only stays bounded for $\omega=0$.
\begin{figure}
\begin{center}
\includegraphics{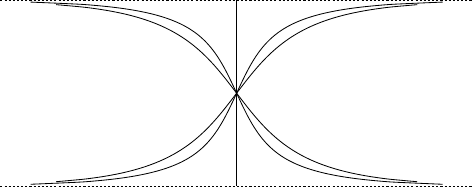}
\end{center}
\caption{Some geodesics in $\AdS'$.}
\label{lem:geodInAdS':fig}
\end{figure}

\end{lem}
\begin{proof}
This follows from transforming the geodesics in \AdS to warped product coordinates.
\end{proof}
We state the following lemma only for referencing and self-containedness of this document:
\begin{lem}[$\tau$ in $\AdS'$]\label{lem:tauInAdS'}
We have that 
\begin{align*}
(t_1,x_1)&\leq_{\AdS'}(t_2,x_2)\Leftrightarrow \\
&\sin(t_1)\sin(t_2)+\cos(t_1)\cos(t_2)\cosh(x_2-x_1)\leq1\text{ and }t_1\leq t_2
\end{align*}
and
\begin{align*}
\tau_{\AdS'}((t_1,x_1),(t_2,x_2))&=\\
\arccos(\sin(t_1)\sin(t_2)&+\cos(t_1)\cos(t_2)\cosh(x_2-x_1))&\text{if }t_1\leq t_2
\end{align*}
\end{lem}

We expect spaces satisfying the assumptions of the Theorem \ref{thm:globalsplitting} to be very similar to $\AdS'=(-\frac{\pi}{2},\frac{\pi}{2})\times_{\cos}\mb{R}$, namely of the form $(-\frac{\pi}{2},\frac{\pi}{2})\times_{\cos} Y$. By Lemma \ref{lem:geodInAdS'} and by Proposition \ref{pop:PropsLwp}.\ref{pop:PropsLwp:geodesics}, we have an explicit description of timelike distance realizers in such spaces.

\begin{lem}[Conformal time in $\AdS'$]
Conformal time for $\AdS'=(-\frac{\pi}{2},\frac{\pi}{2})\times_{\cos}\mb{R}$ is given by
\begin{equation*}
\eta(t)=\log\left(\tan\left(\frac{x}{2}+\frac{\pi}{4}\right)\right)
\end{equation*}
\begin{equation*}
\eta^{-1}(s)= 2 \arctan(e^s) - \frac{\pi}{2}
\end{equation*}
\end{lem}
\begin{dfn}
A Lorentzian warped product $(-\frac{\pi}{2},\frac{\pi}{2})\times_{\cos}Y$ is called the \emph{anti-deSitter suspension} of $Y$ and is the Lorentzian analogue of spherical suspensions in metric spaces. The corresponding Lorentzian warped product comparison space is $\AdS'$.
\end{dfn}

\begin{lem}[The future causal boundary of $\AdS'$ is always at distance $\pi$ from the bottom point]\label{lem:causalBdryAtPi}
Let $p=(t_1,0)\in\AdS'$ and let $\alpha:[0,a)\to\AdS'$ be a future directed causal curve starting at $p$ with $\lim_{\lambda\to a}t(\alpha(\lambda))=\frac{\pi}{2}$. Then for all $\varepsilon>0$ there is a $t_0$ such that $\lim_{\lambda\to a}\tau((t_0,0),\alpha(\lambda))>\pi-\varepsilon$. 
\end{lem}
\begin{proof}
W.l.o.g.\ assume $x(\alpha(\lambda))\geq0$ for all $\lambda$. Note that $\tau((t_0,0),(t,x))$ is monotonically decreasing in $|x|$, thus the worst case is attained when $\alpha$ is null, i.e.\ we can assume $\alpha(t)=\left(t,\int_{t_1}^t\frac{1}{\cos(t)}\diff t\right)$. 

Now we calculate $\lim_{t\to a}\tau((t_0,0),\alpha(t))$: For the calculation, we switch to the $3$-dimensional picture in Def.\ \ref{def:AdS}, there we can use the formula 
\begin{equation*}
\tau(p,q)=\arccos(-\left<p,q\right>)
\end{equation*}
where the inner product is in the ambient semi-Riemannian vector space, see \cite[(2.7)]{tauInModelSpaces}. The point $\lim_{t\to a}\alpha(t)$ exists and is given by the intersection of the two null geodesics:
\begin{equation*}
\{(0,1,0)+s(1,0,1):s\in\mb{R}\} \;\cap\; \{(\cos(t_1),\sin(t_1),0)+s(-\sin(t_1),\cos(t_1),1) :s\in\mb{R}\}
\end{equation*}
yielding the point
\begin{equation*}
\lim\alpha(t)=\left(\frac{1-\sin(t_1)}{\cos(t_1)},1,\frac{1-\sin(t_1)}{\cos(t_1)}\right)
\end{equation*}
so 
\begin{equation*}
\tau((\cos(t_0),\sin(t_0),0),\lim\alpha(t))=\arccos\left(\frac{\cos(t_0)(1-\sin(t_1))}{\cos(t_1)}+\sin(t_0)\right)
\end{equation*}
where we used the inner product formula for $\tau$.

as $t_0\to-\frac{\pi}{2}$, this goes to $\pi$:
\begin{equation*}
\arccos\left(-1\right)=\pi
\end{equation*}
\end{proof}


\section{Rays, lines, co-rays and asymptotes}\label{sec:lines}

\subsection{Timelike co-ray condition}\label{subsec:lines:lines}

In this subsection, we study causal rays and lines and show how to obtain them as limits of causal maximisers. An \AdSn-line is then a line of length $\pi$, parametrized by $\tau$-unit speed on $(-\frac{\pi}{2},\frac{\pi}{2})$. Then, we analyse triangles where one side is a segment on a timelike \AdSn-line and show that the angles adjacent to the \AdSn-line are equal to their comparison angles. This in fact follows from the more general principle that one can stack triangle comparisons of nested triangles with two endpoints on a timelike \AdSn-line. The latter situation arises in the construction of asymptotes, and via the stacking principle, one can show that any future directed and any past directed asymptote from a common point to a given timelike \AdSn-line fit together to give a (timelike) asymptotic \AdSn-line. This approach is nearly equal to the approach in \cite{triSplitting}, where the construction of co-rays and asymptotes is based on \cite{LorToponogovSplitting} and the stacking principle and equality of angles are based on \cite[Lem.\ 10.5.4]{BBI}.

\begin{dfn}[Rays, Lines, \AdSn-lines]\par
Let $(X,d,\ll,\leq,\tau)$ be a Lorentzian pre-length space. A \emph{(future directed) causal ray} is a future inextendible, future directed causal curve $c:I\to X$ that maximises the time separation between any of its points, where $I$ is either a closed interval $[a,b]$ or a half-open interval $[a,b)$. A \emph{(future directed) causal line} is a (doubly) inextendible, future directed causal curve $\gamma:I\to X$ that maximises the time separation between any of its points. Here $I$ can in general be open, closed, or half-open. It is a \emph{\AdSn-line} if $\gamma$ is timelike and $L(\gamma)=\pi$, parametrized in $\tau$-unit speed parametrization on $(-\frac{\pi}{2},\frac{\pi}{2})$. 
\end{dfn}

\begin{rem}
If $X$ is regularly localisable, then any causal ray/line $c$ is either timelike or null by Thm.\ \cite[Thm.\ 3.18]{saemLorLen}.
\end{rem}

\begin{lem}[Limit of distance realizers is distance realizer]\label{lem:limitRealizer}
Let $X$ be a localizable \LpLS and let $\gamma_n:[a_n,b_n]\to X$ be causal distance realizers converging pointwise to a causal curve $\gamma:[a,b]\to X$ (and $a_n\to a$, $b_n\to b$). Then $\gamma$ is a distance realizer, with $L_{\tau}(\gamma)=\lim_nL_{\tau}(\gamma_n)$. In particular, if both $\gamma_n$ and $\gamma$ are timelike and the convergence is uniform, a choice of $\tau$-arclength parametrizations converge to $\gamma$ too.
\end{lem}
\begin{proof}
This is a part of \cite[Thm.\ 2.23]{triSplitting}. 
\end{proof}

Note that if the limit curve is only defined on an open domain, ``we can apply this to each compact subinterval, but the limit curve can lose length''. We will address this in Prop.\ \ref{pop:TCRCHoldsonI(gamma)} and Prop.\ \ref{pop:asymptoticLines}.

\begin{dfn}[Corays, Asymptotes]
\label{def:asymptotes}
Let $c:[0,b)\to X$ be future causal and future inextensible. Let $x_n\to x$ and $t_n\nearrow b$ such that $x_n\in I^-(c(t_n))$. Let $\alpha_n:[0,a_n]\to X$ be a sequence of future directed, maximising causal curves from $x_n$ to $c(t_n)$ in $d$-arclength parametrisation. Any limit curve $\alpha$ of $\alpha_n$ (see \cite[Thm.\ 3.7]{saemLorLen}) is called a \emph{(future) coray} to $c$ at $x$. By Lem.\ \ref{lem:limitRealizer}, $\alpha$ is automatically maximizing, although it might be null. 

If we choose $x_n=x$ constant, $\alpha$ is called a \emph{(future) asymptote}.
\end{dfn}

A lot of the time, we will be working in the \emph{setup of the local splitting}, assuming $X$ to be a connected, regularly localisable, globally hyperbolic Lorentzian length space with proper metric $d$ and global timelike curvature bounded below by $K=-1$ satisfying timelike geodesic prolongation and containing an \AdSn-line $\gamma:(-\frac{\pi}{2},\frac{\pi}{2}) \to X$ in $\tau$-arclength parametrization.

We can construct past and future directed co-rays from all points on $I(\gamma):=I^+(\gamma) \cap I^-(\gamma)$. 

Since maximising causal curves have causal character, a future coray is either timelike or null, but in our case they will always be timelike:

\begin{pop}\label{pop:TCRCHoldsonI(gamma)}
In the setting of the local splitting, all co-rays from points in $I(\gamma)$ are timelike.
\end{pop}
\begin{proof}
Suppose there is a point $p \in I(\gamma)$ such that the claim does not hold at $p$, so w.l.o.g.\ there is a sequence $p_n \to p$, $t_n \to +\frac{\pi}{2}$ and maximal timelike curves $\sigma_n$ from $p_n$ to $y_n:=\gamma(t_n)$ such that $\sigma_n$ converge to a future directed null ray $\sigma$. Choose some $x \in \gamma \cap I^-(p)$ and let $\mu_n$ be maximal timelike curves from $x$ to $p_n$ (assuming $n$ to be large enough, $x\ll p_n$). Applying the limit curve theorem, it is easily seen that (up to a choice of subsequence) the $\mu_n$ converge locally uniformly to a maximising limit causal curve $\mu$ from $x$ to $p$. $\mu$ is timelike since $x \ll p$. Denote by $\gamma_n$ the piece of $\gamma$ that runs between $x$ and $y_n$. Set $a_n:=L_{\tau}(\mu_n)$, $b_n:=L_{\tau}(\sigma_n)$ and $c_n:=L_{\tau}(\gamma_n)$. Let $\beta_n:=\ma_x(\mu_n,\gamma_n)$ and $\theta_n:=\ma_{p_n}(\mu_n,\sigma_n)$. Then $a_n \to a:=\tau(x,p)$ and by the continuity of angles (cf.\ Prop.\ \ref{pop:continuityOfAngles}), $\beta_n \to \beta$, where $\beta$ is the angle between $\mu$ and $\gamma$. Consider the comparison triangles for $\Delta(x,p_n,y_n)$ in \AdS and call its sides $(\overline{\mu}_n,\overline{\sigma}_n,\overline{\gamma}_n)$, then the angle $\overline{\beta}_n$ between $\overline{\mu}_n$ and $\overline{\gamma}_n$ satisfies $\overline{\beta}_n \leq \beta_n$ and similarly $\overline{\theta}_n \geq \theta_n$. Since $\beta_n \to \beta$, there is some $C > 0$ such that $\overline{\beta}_n \leq C$ for all $n$. The hyperbolic law of cosines in $\AdS$ (see Lem.\ \ref{lem:LOC}) gives 
\begin{align*}
	&\cos(b_n)=\cos(a_n)\cos(c_n)+\sin(a_n)\sin(c_n)\cosh(\overline{\beta}_n),\\
	&\cos(c_n)=\cos(a_n)\cos(b_n)-\sin(a_n)\sin(b_n)\cosh(\overline{\theta}_n).
\end{align*}
Using the first equation, noting that $\overline{\beta}_n$ stays bounded and that $a_n$ and $c_n$ stay bounded away from zero and $\pi$, we get that also $b_n$ stays bounded away from zero and $\pi$. This means that also $\overline{\theta}_n$ stays bounded. 

From the monotonicity condition we get that 
\begin{equation*}
\tilde{\ma}_{p_n}(\sigma_n(s),\mu_n(t)))\leq\tilde{\ma}_{p_n}(x,y_n)=\bar{\theta}_n
\end{equation*}
for each $s$ and $t$, so this is bounded too. We will see that this is incompatible with $\sigma_n$ ``getting more and more null'': Note that we get the following estimate for some constant $C''$: 
\begin{align*}
C'' &\geq \cosh( \tilde{\measuredangle}_{p_n}(\sigma_n(s),\mu(t))) \\
&= \frac{\cos(\tau(\mu_n(t),p_n))\cos(\tau(p_n,\sigma_n(s)))-\cos(\tau(\mu_n(t),\sigma_n(s)))}{\sin(\tau(p_n,\sigma_n(s))\sin(\tau(\mu_n(t),p_n))}.
\end{align*}
Since the denominator goes to $0$ for $n \to \infty$ (as $\tau(p_n,\sigma_n(s)) \to \tau(p,\sigma(s)) = 0$ and $\tau(\mu_n(t),p_n) \to \tau(\mu(t),p) > 0$), the enumerator has to go to $0$ as well, in particular this means
\begin{align*}
    \tau(\mu(t),\sigma(s)) = \tau(\mu(t),p)
\end{align*}
for all $s,t$. This implies that running along $\mu$ from $\mu(t)$ to $p$ and then along $\sigma$ to $\sigma(s)$ gives a maximiser, but this curve has a timelike and a null piece, a contradiction to \cite[Thm.\ 3.18]{saemLorLen}.
\end{proof}

\begin{pop}\label{pop:asyInext}
Let $X$ be a $d$-compatible, locally causally closed and causal \LpLS with proper $d$ and $c:[0,b)\to X$ future inextensible, and let $\alpha:[0,a)\to X$ be a timelike asymptote to $c$. Then $\alpha$ is future inextensible.
\end{pop}
\begin{proof}
Let $\alpha_n:[0,a_n]\to X;\,p\leadsto c(t_n)$ be the sequence converging to $\alpha$, it is parametrized in $d$-arclength parametrization. As $d$ is proper, we have that $a_n\to+\infty$\footnote{Otherwise by properness, $c(t_n)$ would converge to some $P_1$, and if the whole curve didn't converge to $P_1$ other subsequence would converge to some other point $P_2$ within a local causal closed neighbourhood of $P_1$, and we get $P_1\leq P_2\leq P_1$, contradicting causality of $X$.}. Now we can apply the Limit curve theorem for inextensible curves (\cite[Thm.\ 3.14]{saemLorLen}) to get that $\alpha$ is inextensible.
\end{proof}

We will be in particular interested in future and past asymptotes of an \AdSn-line $c:(-\frac{\pi}{2},\frac{\pi}{2})\to X$.

We turn to the treatment of triangles adjacent to \AdSn-lines and begin with the aforementioned stacking principle. It will be an essential technical tool for controlling the behaviour of asymptotes to an \AdSn-line. The following two results are true for rather general Lorentzian pre-length spaces, the exact conditions are specified.

\begin{pop}[Comparison situations stack along an \AdSn-line]
\label{pop:compStack}
Let $X$ be a timelike geodesically connected \LpLS with global timelike curvature bounded below by $-1$ and $\gamma:(-\frac{\pi}{2},\frac{\pi}{2})\to X$ be a timelike \AdSn-line. Let $p\in X$ be a point not on $\gamma$. Let $t_1<t_2<t_3$ such that all $y_i:=\gamma(t_i)$ are timelike related to $p$, see Figure \ref{pop:compStack:figStacking_domain}. Let $\bar{\Delta}_{12}:=\Delta(\bar{p},\bar{y}_1,\bar{y}_2)$ be a comparison triangle for $\Delta_{12}:=\Delta(p,y_1,y_2)$ and extend the side $\bar{y}_1,\bar{y}_2$ to a comparison triangle $\bar{\Delta}_{23}:=\Delta(\bar{p},\bar{y}_2,\bar{y}_3)$ for $\Delta_{23}:=\Delta(p,y_2,y_3)$, all in \AdSn. We choose it in such a way that $\bar{y}_1$ and $\bar{y}_3$ lie on opposite sides of the line through $\bar{p},\bar{y}_2$. Then $\bar{y}_1,\bar{y}_2,\bar{y}_3$ are collinear. That makes $\bar{\Delta}_{13}:=\Delta(\bar{p},\bar{y}_1,\bar{y}_3)$ a comparison triangle for $\Delta_{13}:=\Delta(p,y_1,y_3)$. In particular, all this can be realized in $\AdS'$ such that $\bar{y}_i=(t_i,0)$.
\end{pop}
\begin{proof}
We set $s_-=\sup(\gamma^{-1}(I^-(p)))$ and $s_+=\inf(\gamma^{-1}(I^+(p)))$. Then the set of $s$ where $\gamma(s)$ is timelike related to $p$ is $(-\infty,s_-)\cup(s_+,+\infty)$, see Figure \ref{pop:compStack:figStacking_domain}. We assume $p\ll y_2$, the other case $y_2\ll p$ can be reduced to this one by flipping the time orientation of the space.

We create comparison situations for an Alexandrov situation: We take the triangles $\bar{\Delta}_{12}=\Delta(\bar{p},\bar{y}_1,\bar{y}_2)$ and $\bar{\Delta}_{23}=\Delta(\bar{p},\bar{y}_2,\bar{y}_3)$ as in the statement. 
We also create a comparison triangle $\tilde{\Delta}_{13}=\Delta(\tilde{p},\tilde{y}_1,\tilde{y}_3)$ for $(p,y_1,y_3)$ and get a comparison point $\tilde{y}_2$ for $y_2$ on the side $y_1y_3$. Then we can apply curvature comparison to get $\bar{\tau}(\tilde{p},\tilde{y}_2) \geq \tau(p,y_2)=\bar{\tau}(\bar{p},\bar{y}_2)$. By Lemmas \ref{lem:alexlemAcross} and \ref{lem:alexlemFuture}\footnote{If $p\ll y_1\ll y_2\ll y_3$, we use Lem.\ \ref{lem:alexlemFuture} and if $y_1\ll p\ll y_2\ll y_3$ we use Lem.\ \ref{lem:alexlemAcross}.}, this means the situation is convex, i.e.\ 
\begin{equation}\label{pop:compStack:eq:Convex}
\tilde{\ma}_{y_2}(p,y_1)\leq \tilde{\ma}_{y_2}(p,y_3)\,.
\end{equation} 
If $p\ll y_1$ we also get
\begin{equation}\label{pop:compStack:eq:Monotonous}
\tilde{\ma}_{y_1}(p,y_2)\leq\tilde{\ma}_{y_1}(p,y_3)
\end{equation}
and if $y_1\ll p$, inequality (\ref{pop:compStack:eq:Monotonous}) flips.

Note that inserting back $y_2=\gamma(t_2)$ and $y_3=\gamma(t_3)$ in (\ref{pop:compStack:eq:Monotonous}), varying $t_2,t_3$ and replacing $y_1$ by $y_2$ gives that $\tilde{\ma}_{y_2}(p,\gamma(t))$ is monotonically increasing in $t$ for $t>t_2$. Similarly, the time-reversed situation gives that for all $t$ with $p\ll \gamma(t)$, $\tilde{\ma}_{y_2}(p,\gamma(t))$ is monotonically increasing in $t$ and similarly for $t$ with $\gamma(t)\ll p$. Note that inserting the definitions of $y_1$ and $y_3$ in (\ref{pop:compStack:eq:Convex}) and varying $t_1,t_3$ gives us the last piece to say $\tilde{\ma}_{y_2}(p,\gamma(t))$ is monotonically increasing on its whole domain. We revisit (\ref{pop:compStack:eq:Convex}): $\tilde{\ma}_{y_2}(p,\gamma(t_1))\leq \tilde{\ma}_{y_2}(p,\gamma(t_3))$. Note that the left hand side is decreasing with decreasing $t_1$ and the right hand side is increasing with increasing $t_3$.
\begin{figure}
\begin{center}
\includegraphics{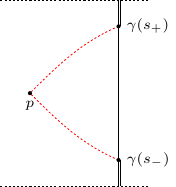}
\end{center}
\caption{The domain where the $y_i$ can lie in is doubled.}
\label{pop:compStack:figStacking_domain}
\end{figure}

\begin{figure}
\begin{center}
\includegraphics{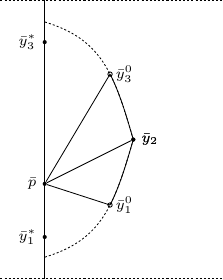}
\end{center}
\caption{We assume that the path $\bar{y}_1^0\bar{y}_2\bar{y}_3^0$ is longer than the path $\bar{y}_1^0\bar{p}\bar{y}_3^0$. If the angle at $\bar{y}_2$ is not straight, we extend some lines and as points further from $\bar{p}$ are more on the left, we find some $t_1^*$ and $t_3^*$ such that $\bar{y}_1^*\bar{p}\bar{y}_3^*$ are collinear. But then the path $\bar{y}_1^0\bar{y}_2\bar{y}_3^0$ is shorter than $\bar{y}_1^0\bar{p}\bar{y}_3^0$, which then yields a contradiction to $\gamma$ being maximising.}
\label{pop:compStack:figMainargument}
\end{figure}

\underline{Claim}: For each $t_1$ and $t_3$, equality holds in (\ref{pop:compStack:eq:Convex}).

\underline{Then}: The comparison situation is straight, i.e.\ $\bar{y}_1,\bar{y}_2,\bar{y}_3$ lie on a distance realizer, and by triangle equality along distance realizers we have $\tau(y_1,y_3)=\tau(y_1,y_2)+\tau(y_2,y_3)=\bar{\tau}(\bar{y}_1,\bar{y}_2)+\bar{\tau}(\bar{y}_2,\bar{y}_3)=\bar{\tau}(\bar{y}_1,\bar{y}_3)$, i.e., $\bar{p},\bar{y}_1,\bar{y}_3$ is a comparison triangle for $p,y_1,y_3$.

\underline{Proof}: For any $t_1,t_3$ we can draw $\bar{\Delta}_{12}$, $\bar{\Delta}_{23}$ simultaneously in \AdSn. We indirectly assume there are $t_1,t_3$ such that the inequality (\ref{pop:compStack:eq:Convex}) is strict. 
We first want to show that there are such $t_1,t_3$ where $\bar{\Delta}_{12}$, $\bar{\Delta}_{23}$ can be drawn simultaneously in the warped product picture $\AdS'$. As a first case, let inequality (\ref{pop:compStack:eq:Convex}) always be strict. Then we take the limit as $t_1\nearrow t_2$ and $t_3\searrow t_2$, as the sides are monotonous the limits of the sides are both finite, thus for $t_1,t_3$ close enough to $t_2$, we have that $\bar{y}_1,\bar{y}_3$ are realizable in $\AdS'$ (by triangle equality of angles in \AdSn, we take an angle a little bit bigger than the limits, then small distances can be realized at that angle or lower from a prescribed side $\bar{p}\bar{y}_2$). The other case is where inequality (\ref{pop:compStack:eq:Convex}) has equality for $t_1,t_3$ close to $t_2$, but one of the sides changes when taking $t_1$ resp.\ $t_3$ further away from $t_2$. W.l.o.g., say the right side changes, so there is a $t_3^0>t_3$ such that $\tilde{\ma}_{y_2}(p,\gamma(t_3^0))>\tilde{\ma}_{y_2}(p,\gamma(t_3))$. Note that as long as (\ref{pop:compStack:eq:Convex}) has equality, we have that $\tau(\bar{y}_1,\bar{y}_3)=\tau(\bar{y}_1,\bar{y}_2)+\tau(\bar{y}_2,\bar{y}_3)<L(\gamma)=\pi$. As $\tilde{\ma}_{y_2}(p,\gamma(t_3))$ is continuous in $t_3$, there also exists a $t_3^0$ such that $\tilde{\ma}_{y_2}(p,\gamma(t_3^0))>\tilde{\ma}_{y_2}(p,\gamma(t_3))$ and $\tau(\bar{y}_1,\bar{y}_3)<\pi$. In particular, it is realizable, with $\bar{y}_1^0,\bar{y}_3^0$ on a vertical line, and all of $\bar{y}_1^0,\bar{y}_3^0,\bar{p}$ to the left of the vertical line with $x$-coordinate $x(\bar{y}_2)$. In any case, we now fix this realization, i.e.\ fix the points $\bar{y}_1^0,\bar{y}_2,\bar{y}_3^0$ and $\bar{p}$. 

Now consider $t_1^*<t_1^0$ and $t_3^*>t_3^0$. On the side $\bar{y}_2\bar{p}$, we can construct comparison points for $\gamma(t_1^*)$ and $\gamma(t_3^*)$, call them $\bar{y}_1^*,\bar{y}_3^*$, if they are still contained in $\AdS'$. Note that $\bar{y}_1^*,\bar{y}_3^*$ vary continuously with $t_1^*,t_3^*$, so if $\bar{y}_1^*,\bar{y}_3^*$ are not defined for some $t_i^*$ we get a $t_1^{\min},t_3^{\max}$ such that it is defined before that value, and $\bar{y}_1^*$ or $\bar{y}_3^*$ converge to infinity in the limit to $t_1^{\min}$ resp.\ $t_3^{\max}$, otherwise set $t_1^{\min}=-\frac{\pi}{2}$ and / or $t_3^{\max}=\frac{\pi}{2}$, then similar limits hold. As $\tilde{\ma}_{y_2}(p,\gamma(t_3^*))>\tilde{\ma}_{y_2}(p,\gamma(t_3^0))$, we get that $\bar{y}_3^*$ is to the left of the (extended) geodesic connecting $\bar{y}_2t_3^0$, and as $\tau(\bar{y}_2,\bar{y}_3^*)>\tau(\bar{y}_2,\bar{y}_3^0)$, it has much more negative $x$-coordinate. As the (extended) geodesic connecting $\bar{y}_2t_3^0$ is going to the left, it has $x$-coordinate converging to $-\infty$ towards the boundary, and thus so has $\bar{y}_3^*$, similarly $\bar{y}_1^*$. In particular, as $\bar{y}_1^*,\bar{y}_3^*$ depend continuously on $t_1^*,t_3^*$, there are parameters such that $\bar{y}_1^*,\bar{y}_3^*$ have the same $x$-coordinate as $p$, thus by triangle equality along distance realizers and strict triangle inequality we have that
\begin{equation*}
\tau(\bar{y}_1^*,\bar{p})+\tau(\bar{p},\bar{y}_3^*)=\tau(\bar{y}_1^*,\bar{y}_3^*)>\tau(\bar{y}_1^*,\bar{y}_2)+\tau(\bar{y}_2,\bar{y}_3^*)\,,
\end{equation*}
in contradiction to the fact that $\gamma(t_1^*),\gamma(t_2),\gamma(t_3^*)$ lie on a distance realizer. Thus we get the claim. See Figure \ref{pop:compStack:figMainargument} for a visualisation of the construction.
\end{proof}

\begin{pop}[Angle = comparison angle]
\label{pop:angle=comparisonangle}
Let $X$ be a timelike geodesically connected \LpLS with global timelike curvature bounded below by $K=-1$ and $\gamma:(-\frac{\pi}{2},\frac{\pi}{2})\to X$ be a complete timelike \AdSn-line and $x:=\gamma(t_0)$ a point on it. We split $\gamma$ into the future part $\gamma_+=\gamma|_{[t_0,\frac{\pi}{2})}$ and the past part $\gamma_-=\gamma|_{(-\frac{\pi}{2},t_0]}$. Let $p\in X$ be a point not on $\gamma$ with $x$ and $p$ timelike related and $\alpha:x\leadsto p$ a connecting distance realiser. Then for all $s\neq t$ such that $p$ and $\gamma(s)$ are timelike related, we have:
\begin{equation*}
\tilde{\ma}_x(p,\gamma(s))=\ma_x(\alpha,\gamma_+)=\ma_x(\alpha,\gamma_-)\,,
\end{equation*}
i.e.\ the comparison angle is equal to the angle, and the same in both directions.
\begin{figure}
\begin{center}
\includegraphics{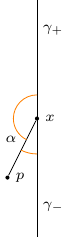}
\end{center}
\caption{Illustration: These angles are the same, and have the same value as if they are considered as comparison angles. }
\label{pop:angle=comparisonangle:figStacking_angles}
\end{figure}
\end{pop}
\begin{proof}
First, we check that the comparison angle $\tilde{\ma}_x(p,\gamma(s))$ is constant in $s$: For $s_1$ and $s_2$ for which this is defined, we have three parameters on $\gamma$ involved: $s_1,s_2,t_0$. The previous result (Prop.\ \ref{pop:compStack}) tells us we can construct a comparison situation for all three triangles at once. As the comparison situations have the comparison angle $\tilde{\ma}_x(p,\gamma(s_1))$ resp.\ $\tilde{\ma}_x(p,\gamma(s_2))$ as the angle in $\bar{x}$, they are equal.

Now we look at the angle $\ma_x(\alpha,\gamma_\pm)$: We assume $p\ll x$. We already know $\tilde{\ma}_x(\alpha(s),\gamma(t))$ is constant in $t$. We now have to look at its dependence on $s$: By Thm.\ \ref{thm:equivTriCompAndMonComp}, $\tilde{\ma}_x^{\mathrm{S}}(\alpha(s),\gamma(t))$ is monotonically increasing in $s$. Note now that for $t<t_0$, the sign of this angle is $\sigma=-1$, and for $t>t_0$, the  sign of this angle is $\sigma=+1$. So choose some $t_-<t_0$ and $t_+>t_0$ for which all the necessary angles exist (i.e.\ $\gamma(t_-)\ll p$ and $t_+>t_0$). Then the above applied to $\alpha(s)$ instead of $p$ gives that $\tilde{\ma}_x(\alpha(s),\gamma(t_-))=\tilde{\ma}_x(\alpha(s),\gamma(t_+))$ is both a monotonically decreasing and increasing function in $s$. Thus it is constant, and $\tilde{\ma}_x(\alpha(s),\gamma(t))=\ma_x(\alpha,\gamma_-)=\ma_x(\alpha,\gamma_+)$, which includes the desired equalities.
\end{proof}
Note the situation we are in is very rigid: as soon as a side of a triangle is part of an \AdSn-line, the curvature bound inequality of one-sided triangle comparison (different version of triangle comparison, see e.g \cite[Def.\ 3.2]{equivDefs}) where the point is on the line $\gamma$ automatically has equality:
\begin{cor}\label{cor:sides_equal}
Let $X$ be a timelike geodesically connected, globally causally closed Lorentzian pre-length space with global timelike curvature bounded below by $K=-1$ and let $\gamma:\R \to X$ be a timelike \AdSn-line. Then for any point $p\in X$ and two points $x_1=\gamma(t_1)$ and $x_2=\gamma(t_2)$ which are timelike related to $p$, we get an unordered timelike triangle $\Delta=\Delta(x_1,x_2,p)$. Let $q_1,q_2$ be any points on $\Delta$, one of them lying on the side $x_1x_2$. We form a comparison triangle $\bar{\Delta}$ for $\Delta$ and find comparison points $\bar{q}_1,\bar{q}_2$ for $q_1,q_2$. Then $q_1\leq q_2$ if and only if $\bar{q}_1\leq\bar{q}_2$; and $\tau(q_1,q_2)=\bar{\tau}(\bar{q}_1,\bar{q}_2)$.
\end{cor}
\begin{proof}
We distinguish which sides the $q_i$ lie on: Note we assumed one of them is on the side $x_1x_2$. We only prove the case where $q_1$ is on the side $\gamma_+$ connecting $x_1x_2$ (say $q_1=\gamma_+(s)$) and $q_2$ is on the side $\beta$ connecting $x_1,p$ (say $q_2=\beta(t)$), the proof of the other cases is easily adapted.

Now the comparison triangle $\tilde{\Delta}$ for $\Delta(x_1,q_1,q_2)$ has angle $\tilde{\ma}_{x_1}(q_1,q_2)=\ma_{x_1}(\gamma_+,\beta)$ at the point corresponding to $x_1$, and the comparison triangle for $\Delta=\Delta(x_1,p,x_2)$ has the same angle $\tilde{\ma}_{x_1}(p,x_2)=\ma_{x_1}(\gamma_+,\beta)$ at the point corresponding to $x_1$. As the comparison points $\bar{q}_1$ and $\bar{q}_2$ lie on the sides enclosing the angle, also $\Delta(\bar{x}_1,\bar{q}_1,\bar{q_2})$ has the angle $\ma_{x_1}(\gamma_+,\beta)$ at $\bar{x}_1$. Now the triangles $\tilde{\Delta}$ and $\Delta(\bar{x}_1,\bar{q}_1,\bar{q_2})$ have two sides and the angle between them equal, thus also the opposite side is equal, establishing $\tau(q_1,q_2)=\bar{\tau}(\bar{q}_1,\bar{q}_2)$. A continuity argument (moving $q_1$ a bit into the past) then gives that $q_1\leq q_2$ if and only if $\bar{q}_1\leq\bar{q}_2$.
\end{proof}

In the following, unless we specify the assumptions on $X$, we always take $X$ to be as in the local splitting.

Although in the model for this paper, the splitting theorem, it was possible to circumvent the Busemann functions, this is not possible here: In the warped product the Busemann function becomes the $t$-coordinate and as the radius function varies, this is important when making comparison situations.

\begin{dfn}
We define the \emph{past Busemann function} w.r.t.\ $\gamma$, $b_-:I(\gamma)\to (-\frac{\pi}{2},\frac{\pi}{2})$ as $b_-(x)=\frac{\pi}{2}-\lim_{t\to-\frac{\pi}{2}}\tau(\gamma(t),x) = \lim_{t\to\frac{\pi}{2}}(t-\tau(\gamma(-t),x))$. Similarly, the \emph{future Busemann function} w.r.t.\ $\gamma$ is defined by setting $b_+(x)=+\frac{\pi}{2}-\lim_{t\to+\frac{\pi}{2}}\tau(x,\gamma(t))=\lim_{t\to+\frac{\pi}{2}}(t-\tau(x,\gamma(t)))$. Note that the arguments in both limits are monotonically decreasing.
\end{dfn}

\begin{lem}
We have that $b_++b_-\geq0$.
\end{lem}
\begin{proof}
Let $x\in I(\gamma)$. We dive into the limits: We have that for $t_1$ small enough and $t_2$ large enough that 
\begin{equation*}
\tau(\gamma(t_1),\gamma(t_2))-\pi\geq \underbrace{\tau(\gamma(t_1),x)-\frac{\pi}{2}}_{\to-b_-(x)}+\underbrace{\tau(x,\gamma(t_2))-\frac{\pi}{2}}_{\to-b_+(x)}\,,
\end{equation*}
and that the left hand side converges to $0$ as $t_1\to-\frac{\pi}{2}$ and $t_2\to\frac{\pi}{2}$.
\end{proof}

Next, we show that any future directed asymptote to $\gamma$ (which we now know to be timelike) has the right length.

\begin{pop}\label{pop:asyRightLength}
Let $x \in I(\gamma)$ and let $\alpha$ be a future directed timelike asymptotic ray to $\gamma$ at $x$, and let $t=b_+(x)$. Then $L_{\tau}(\alpha) = \frac{\pi}{2}-t$. 
\end{pop}
\begin{proof}
By construction, $\alpha:[0,b) \to X$ arises as a locally uniform limit of timelike maximisers $\alpha_n:[0,a_n] \to X$ from $x$ to $z_n:=\gamma(t_n)$, where $t_n \to +\frac{\pi}{2}$. For each compact subinterval of the domain of $\alpha$ we can apply Lem.\ \ref{lem:limitRealizer} to get $L(\alpha|_{[0,c]})=\lim_nL(\alpha_n|_{[0,c]})\leq \lim_n L(\alpha_n)$, thus $L:=L_{\tau}(\alpha) \leq \frac{\pi}{2}-t$. Now indirectly suppose that $L:=L_{\tau}(\alpha) < \frac{\pi}{2}-t$.  By continuity of angles (see Prop.\ \ref{pop:continuityOfAngles}), $\omega_n:=\ma_x(\alpha,\alpha_n) \to 0$ for $n \to \infty$. 

Now note that for any $t_n$ with $n \geq N$, $\partial J^-(z_n) \cap \alpha = (J^-(z_n)\setminus I^-(z_n)) \cap \alpha$ is non-empty: Certainly, $x \in I^-(z_n)$, so if this intersection were empty, then $\alpha$ would be imprisoned in the compact set $J^+(x) \cap J^-(z_n)$, which cannot happen by Prop.\ \ref{pop:asyInext}. So we find a point $y_n \in \alpha$ that is null-related to $z_n$, i.e. $y_n < z_n$ and $\tau(y_n,z_n) = 0$.

We now consider the triangle given by the vertices $x,y_n,z_n$. Two of the sides have natural curves: From $x$ to $y_n$ set $\nu_n$ to be the part of $\alpha$ from $x$ to $y_n$, and from $x$ to $z_n$ take $\alpha_n$. We name the side-lengths: $c_n=\tau(x,z_n)=L(\alpha_n)$, $a_n=\tau(x,y_n)=L(\nu_n)$ and $b_n=\tau(y_n,z_n)=0$. We know a lot about this triangle as $n\to+\infty$:
\begin{itemize}
\item By the definition of $t=b_+(x)$, $c_n\to \frac{\pi}{2}-t$. Note this is the longest side-length and $c_n$ is bounded away from $\pi$ as $b_+(x)\geq -b_-(x)>-\frac{pi}{2}$.
\item $b_n=0$ by the choice of $y_n$.
\item $a_n=\tau(x,y_n)<L(\alpha)$, so $\limsup_{n} a_n \leq L<\frac{\pi}{2}-t$.
\item Our curvature assumption gives that $\bar{\omega}_n:=\tilde{\ma}_x(y_n,z_n)\leq \omega_n\to 0$.
\end{itemize}
But this cannot be: we claim we can take a comparison situation $\Delta(\bar{x},\bar{y}_n,\bar{z}_n)$ which converges as $n\to+\infty$: Fix $\bar{x}$, choose $\bar{z}_n$ along a fixed distance realizer $\bar{\alpha}_n$ (independent of $n$!) through $\bar{x}$, take $\bar{y}_n$ on a distance realizer $\bar{\alpha}$ (dependent on $n$!) through $\bar{x}$ at the correct angle $\ma_{\bar{x}}(\bar{\alpha}_n,\bar{\alpha})=\bar{\omega}_n$. Then $\bar{x}$ and $\bar{z}_n$ automatically converge, and $\bar{y}_n$ does so too as $\bar{\omega}_n$ converges. 

Note the limiting situation is degenerate, as $\bar{\omega}_n\to0$, $c_n$ is bounded away from $0$ and $a_n>0$ is increasing, so bounded away from $0$ as well. Thus, the triangle inequality of angles has equality in the limit, so
\begin{equation*}
0=\lim_n (c_n-a_n-b_n) = \left(\frac{\pi}{2}-t\right) - L - 0
\end{equation*}
which contradicts $L<\frac{\pi}{2}-t$.
\end{proof}

To conclude this subsection, we show that any future directed and any past directed (timelike) asymptote to $\gamma$ from a common point fit together to give an \AdSn-line.

\begin{pop}[Asymptotic \AdSn-lines]
\label{pop:asymptoticLines}
Let $p \in I(\gamma)$ and consider any future and past asymptotes $\sigma^+:[0,a_+) \to X$ and $\sigma^-:(a_-,0] \to X$ from $p$ to $\gamma$. Then $\sigma :=\sigma^- \sigma^+: (a_-,a_+) \to X$ is a complete, timelike distance realizer of length $\pi$. In particular $b_++b_-=0$ and it can be reparametrized in $\tau$-arclength to be an \AdSn-line. The reparametrization is called a (full) \emph{asymptote} $\sigma$ to $\gamma$ through $p$. In particular, for an asymptote $\sigma$ we have $b_+(\sigma(t))=t$ ("the Busemann parametrization is a $\tau$-arclength parametrization"). 
\end{pop}
\begin{proof}
Let $\sigma_n^+$ and $\sigma_n^-$ be two sequences of timelike maximisers from $p$ to $\gamma(r_n)$ and $\gamma(-r_n)$, respectively, such that $\sigma^+$ and $\sigma^-$ arise as limits of these sequences as $r_n \to \infty$. To show that $\sigma$ is a line, it is sufficient to show that for any $s < 0<t$, $\tau(\sigma(s),\sigma(t)) = L_{\tau}(\sigma|_{[s,t]})$. To see this, let $q_+:=\sigma(t)$ and $q_-:=\sigma(s)$, and $q_+^n:=\sigma_n^+(t)$, $q_-^n:=\sigma_n^-(s)$. Then $q_{\pm} = \lim_n q_{\pm}^n$. Consider the triangle $\Delta(q_-^n,p,q_+^n)$ with sides $\sigma_n^-$, $\sigma_n^+$ and a part of $\gamma$. Consider a comparison triangle with points $\overline{q_{\pm}^n}$ corresponding to $q_{\pm}^n$. Sending $n \to \infty$, we see that the stacked comparison triangles in $\AdS'$ converge to a vertical line (here we use Prop.\ \ref{pop:compStack} and that $\gamma$ has length $\pi$), hence our curvature bound gives
\begin{align*}
    \tau(q_-,q_+) = \lim_{n \to \infty} \tau(q_-^n,q_+^n) \leq \lim_{n \to \infty} \overline{\tau}(\overline{q_-^n},\overline{q_+^n}) = \lim_n L_{\tau}(\sigma_n^-\sigma_n^+|_{[s,t]})= L_{\tau}(\sigma|_{[s,t]}),
\end{align*}
which is what we wanted to show, as the other inequality is trivial.

$a_+-a_-=\pi$ also follows as the stacked comparison triangles in $\AdS'$ converge to a vertical line. For the future and past Busemann functions, we also look at this picture: we get a vertical line $\bar{\gamma}:(-\frac{\pi}{2},\frac{\pi}{2})\to\AdS'$ and $\bar{p}\in\AdS'$ such that any $\Delta(\bar{\gamma}(t_1),\bar{p},\bar{\gamma}(t_2))$ is a comparison triangle. In particular, $b_+(p)=\frac{\pi}{2}-\lim_{t_2\to\frac{\pi}{2}}\bar{\tau}(\bar{p},\bar{\gamma}(t_2)) = t(p) =\cdots= -b_-(p)$ is the $t$-coordinate.
\end{proof}

\subsection{\texorpdfstring{\AdSn-Parallel lines}{AdS-Parallel lines}}
\label{subsec:lines:parallel}

This subsection introduces the notion of parallelism for complete timelike \AdSn-lines. In the following, we call a map $f:Y_1 \to Y_2$ between Lorentzian pre-length spaces $(Y_1,d_1,\ll_1,\leq_1,\tau_1)$ and $(Y_2,d_2,\ll_2,\leq_2,\tau_2)$ $\tau$-\emph{preserving} if for all $p,q \in Y_1$, $\tau_1(p,q) = \tau_2(f(p),f(q))$, and we call $f$ $\leq$-\emph{preserving} if $p \leq_1 q$ if and only if $f(p) \leq_2 f(q)$.

\begin{dfn}[Realizing w.r.t.\ Busemann]
Let $\gamma$ be a \AdSn-line and $x\in I(\gamma)$. Then a point $\bar{x}\in\AdS'$ is  \emph{future-Busemann-realized} if $t(\bar{x})=b_+(x)$. Similarly, for an asymptote $\alpha:I\to X$ to $\gamma$ parametrized in $\tau$-arclength parametrization, $\bar{\alpha}:I\to\AdS'$ is \emph{Busemann-realized} if it is vertical (i.e.\ $(\bar{\alpha}(t))_x$ is constant) and $t(\bar{\alpha}(t))=b_+(\alpha(t))$. Note that by Prop.\ \ref{pop:asymptoticLines}, this is always possible, and $\alpha$ is in $\tau$-arclength parametrization.
\end{dfn}

\begin{dfn}[\AdSn-Parallel lines]
\label{def:parallelLines}
Let $\alpha,\beta$ be two \AdSn-lines in a Lorentzian pre-length space $X$. They are called \emph{(\AdSn-)parallel} if there exists a $\tau$- and $\leq$-preserving map $f:(\alpha((-\frac{\pi}{2},\frac{\pi}{2}))\cup \beta((-\frac{\pi}{2},\frac{\pi}{2})))\to (-\frac{\pi}{2},\frac{\pi}{2})\times_{\cos}\mb{R}$ such that $f\circ\alpha$ and $f\circ\beta$ are Busemann-realizations of $\alpha$ and $\beta$. We call such a map $f$ a \emph{(\AdSn-)parallel realisation} of $\alpha$ and $\beta$. The \emph{spacelike distance} between $\alpha$ and $\beta$ is the $x$-coordinate difference in the realization. Note that one can similarly define the spacelike distance between an \AdSn-line and a point (with the stacking principle \ref{pop:compStack} ensuring that it always works out).
\end{dfn}

\begin{rem}
Note that by post-composing this by a translation in the $x$-direction, we can always achieve that $f(\alpha(\R))=\{(t,0):t\in\R\}\subseteq\AdS'$ and $f(\beta(\R))=\{(t,c):t\in\R\}$ for some $c\geq0$.
\end{rem}
\begin{dfn}
Let $X$ be a \LpLS with continuous time separation $\tau$ satisfying $\tau(x,x) = 0$ for all $x \in X$. Let $\alpha,\beta:(-\frac{\pi}{2},\frac{\pi}{2})\to X$ be two $\tau$-arclength parametrised \AdSn-lines. 
For each $s<t$ with $\alpha(s)\leq\beta(t)$ there is a $c\geq0$ such that the \AdSn-parallel lines $\bar{\alpha},\bar{\beta}:(-\frac{\pi}{2},\frac{\pi}{2})\to\AdS'$ with distance $c$ so that the partial map $\alpha(s)\mapsto\bar{\alpha}(s)$ and $\beta(t)\mapsto\bar{\beta(t)}$ is (Busemann- and) $\tau$-preserving. Set:
\begin{itemize}
\item $c_{\alpha\beta}(s,t)$ to be that $c$,
\item $c_{\beta\alpha}(s,t)$ to be the analogous thing with $\alpha(s)$ and $\beta(t)$ reversed.
\end{itemize}
Similarly, for each $s$, we find a $c\geq0$ such that the closure of the set $\{t:\alpha(s)\leq\beta(t)\}$ is the same as $\{t:\bar{\alpha}(s)\leq\bar{\beta}(t)\}$. Set:
\begin{itemize}
\item $c_{\alpha+}^N(s)$ to be that $c$,
\item $c_{\beta+}^N(s)$ to be the analogous thing with $\alpha(s)$ and $\beta(t)$ reversed.
\end{itemize}
By Lem.\ \ref{lem:tauInAdS'}, we get explicit formulae:

\begin{itemize}
\item $c_{\alpha\beta}(s,t)=\arcosh\left(\frac{\cos(\tau(\alpha(s),\beta(t)))-\sin(s)\sin(t)}{\cos(s)\cos(t)}\right)$ if $\alpha(s)\leq\beta(t)$ (otherwise undefined, note the formula is strictly decreasing in $\tau(\alpha(s),\beta(t))$)),
\item $c_{\beta\alpha}(s,t)=\arcosh\left(\frac{\cos(\tau(\beta(t),\alpha(s)))-\sin(s)\sin(t)}{\cos(s)\cos(t)}\right)$ if $\beta(t)\leq\alpha(s)$ (otherwise undefined), 
\item $c_{\alpha+}^N(s)=\arcosh\left(\frac{1-\sin(s)\sin(t)}{\cos(s)\cos(t)}\right)$ for $t=\inf\{t:\alpha(s)\leq\beta(t)\}$ (note the formula is strictly increasing in $t$),
\item $c_{\beta+}^N(s)=\arcosh\left(\frac{1-\sin(s)\sin(t)}{\cos(s)\cos(t)}\right)$ for $t=\inf\{t:\beta(s)\leq\alpha(t)\}$. 
\end{itemize}
We say the \emph{$c$-criterion for parallel \AdSn-lines} is satisfied if all these are constant where defined, have the same value, and the infima are minima.
\end{dfn}

The following is easy to see from the fact that most timelike geodesics have the $x$-coordinates go to infinity before they leave $\AdS'$ (see Lem.\ \ref{lem:geodInAdS'}):
\begin{lem}[Communicability]\label{lem:communicability}
Let $X$ be globally causally closed and $\alpha,\beta$ be \AdSn-parallel. Then for each $s\in(-\frac{\pi}{2},\frac{\pi}{2})$ there is a $t\in(-\frac{\pi}{2},\frac{\pi}{2})$ such that $\alpha(s)\ll\beta(t)$ (and similarly a $t$ with $\alpha(s)\gg\beta(t)$).
\end{lem}
\begin{proof}
We consider the sets $I=\{(s,t):\alpha(s)\ll\beta(t)\}$ (open) and $J=\{(s,t):\alpha(s)\leq\beta(t)\}=\bar{I}$ (as $X$ is globally causally closed). They are non-empty: as $\beta\cap I^+(\alpha)\neq\emptyset$, we find $\alpha(s_0)\ll\beta(t_0)$, so $(s_0,t_0)\in I$. We claim that the first projection of $I$ is full, i.e.\ $pr_1(I)=\{s:\exists t:\alpha(s)\ll\beta(t)\}=(-\frac{\pi}{2},\frac{\pi}{2})$. We know $s_0$ is in the left hand side, and with it all smaller values $s<s_0$. So let indirectly $s_n\nearrow s_+<\frac{\pi}{2}$ and $t_n\nearrow \frac{\pi}{2}$ such that $(s_n,t_n)\in I$, but $s_+\not\in pr_1(I)$. Then by continuity of $\tau$, $\tau(\alpha(s_n),\beta(t_n))\to0$. Looking at the formula for $c_{\alpha\beta}(s_n,t_n)$, one sees that the enumerator stays away from $0$, whereas the denominator approaches $0$. Thus $c_{\alpha\beta}(s_n,t_n)\nearrow+\infty$ which cannot be as it is constant, so $pr_1(I)=(-\frac{\pi}{2},\frac{\pi}{2})$. 
\end{proof}

\begin{lem}[The $c$-criterion for \AdSn-parallel lines]\label{lem:c-cri}
Let $X$ be a \LpLS with continuous time separation $\tau$ satisfying $\tau(x,x) = 0$ for all $x \in X$. Let $\alpha,\beta:(-\frac{\pi}{2},\frac{\pi}{2})\to X$ be two $\tau$-arclength parametrised \AdSn-lines. Then they satisfy the $c$-criterion for parallel \AdSn-lines if and only if they are parallel \AdSn-lines. If this is the case, the constant $c$ is the distance between $\alpha$ and $\beta$.

If $X$ is additionally globally causally closed and $\alpha \cap I^+(\beta) \neq \emptyset$ and $\beta \cap I^+(\alpha) \neq \emptyset$, the $c$-criterion simplifies to only checking $c_{\alpha\beta}$ and $c_{\beta\alpha}$ for being constant and equal.
\end{lem}
\begin{proof}
We define $f$ as in the definition of parallel \AdSn-lines with distance $c$: $f(\alpha(s)):=(s,0)$, $f(\beta(s)):=(s,c)$ in $\AdS'$. We will prove that $f$ is $\leq$-preserving if and only if $c_{\alpha+}^N$ and  $c_{\beta+}^N$ are constantly $c$, and under the assumption that this is the case $f$ is $\tau$-preserving if and only if $c_{\alpha\beta}$ and $c_{\beta\alpha}$ are constantly $c$ wherever defined.

As $\alpha$ and $\beta$ are future directed $\tau$-arclength parametrised \AdSn-lines, it is clear that $f$ is $\leq$- and $\tau$-preserving along $\alpha$ and along $\beta$, i.e.\ we only have to check the conditions on $f(\alpha(s))\mathrel{\bar{\leq}} f(\beta(t))$ and their $\tau$-distance, and both also for $\alpha,\beta$ reversed.

So first, for $\leq$-preserving: For a fixed $s$, we describe the set of $t$ such that $\alpha(s)\leq\beta(t)$ resp.\ $f(\alpha(s))\mathrel{\bar{\leq}} f(\beta(t))$, the situation switching $\alpha$ and $\beta$ is analogous (using $c_{\beta+}^N$ instead of $c_{\alpha+}^N$). By transitivity of $\leq$ resp.\ $\bar{\leq}$, these sets will both be an interval of the form $[t_{min},+\frac{\pi}{2})$ or $(t_{min},+\frac{\pi}{2})$, resp.\ $[\tilde{t}_{min},+\frac{\pi}{2})$ (always closed). Note that inserting $\tilde{t}_{min}$ instead of $t$ into the definition of $c_{\alpha+}^N(s)$ yields $c$ by the geometric definition. As the formula of $c_{\alpha+}^N(s)$ is strictly increasing in $t$, we have that $t_{min}=\tilde{t}_{min}$ if and only if $c=c_{\alpha+}^N(s)$. Thus $\{t:\alpha(s)\leq\beta(t)\}=[\tilde{t}_{min},+\frac{\pi}{2})$ if and only if $c_{\alpha+}^N(s)=c$ and the infimum in the definition of $c_{\alpha+}^N(s)$ is achieved. The same argument works out for $\alpha$ and $\beta$ exchanged, thus proving that $f$ is $\leq$-preserving if and only if $c_{\alpha+}^N$ and $c_{\beta+}^N$ are constant and equal to $c$, and the infima appearing are minima.

For $\tau$-preserving, we assume that $f$ is already $\leq$-preserving. We will only prove $\tau(\alpha(s),\beta(t))=\bar{\tau}(f(\alpha(s)),f(\beta(t)))$ for $\alpha(s)\leq\beta(t)$, the case where they are causally unrelated is covered by $\leq$-preserving, and the case where $\alpha$ and $\beta$ are interchanged follows analogously. 
Note that inserting $\bar{\tau}(f(\alpha(s)),f(\beta(t)))$ for $\tau(\alpha(s),\beta(t))$ into the definition of $c_{\alpha\beta}(s,t)$ yields $c$ by the geometric definition. As the formula of $c_{\alpha\beta}(s,t)$ is strictly decreasing in $\tau(\alpha(s),\beta(t))$, we have that $\tau(\alpha(s),\beta(t))=\bar{\tau}(f(\alpha(s)),f(\beta(t)))$ if and only if $c_{\alpha\beta}(s,t)=c$. The same argument works out for $\alpha$ and $\beta$ exchanged, thus proving that if $f$ is $\leq$-preserving, $f$ is $\tau$-preserving if and only if $c_{\alpha\beta}$ and $c_{\beta\alpha}$ are constant and equal to $c$.

For the additional statement, fix $s$. By global causal closedness, the infimum in the formula of $c_{\alpha+}^N$ is automatically a minimum, and it is non-void by Lem.\ \ref{lem:communicability}. For the value of $c_{\alpha+}^N(s)$, let $t_0$ be the minimum. Then for any $t>t_0$ we have $\alpha(s)\ll\beta(t)$, and by continuity of $\tau$ we have that $\tau(\alpha(s),\beta(t_0))=0$. Now note that plugging this into $c_{\alpha\beta}(s,t_0)$ gives the formula for $c_{\alpha+}^N$.
\end{proof}

\begin{lem}[The strong causality trick]\label{lem:strongCausTrick}
Let $X$ be a strongly causal \LpLS with $\tau$ continuous on $U$ and $\tau(x,x)=0$ for all $x \in X$. Let $\alpha,\beta:[0,b)\to U$ be two $\tau$-arclength parametrised timelike distance realisers with $x:=\alpha(0)=\beta(0)$. Assume that for all $s,t$ such that $\alpha(s)$ and $\beta(t)$ are timelike related, the comparison angle $\tilde{\ma}_x(\alpha(s),\beta(t))=0$. Then $\alpha=\beta$.
\end{lem}
\begin{proof}
We set $f_+(s,t)=\tau(\alpha(s),\beta(t))$ and $f_-(s,t)=\tau(\beta(s),\alpha(t))$. They are both monotonically increasing in $t$ and continuous. We want to describe the set where $f_+>0$ resp.\ $f_->0$. We set $t_s^+=\inf \{t:f_+(s,t)>0\}$, then $\lim_{t\searrow t_s^+} f_+(s,t)=0$. Thus in the law of cosines (Lem.\ \ref{lem:LOC}), we get $\lim_{t\searrow t_s^+}\cosh(\tilde{\ma}_x(\alpha(s),\beta(t)))=1 = \lim_{t\searrow t_s^+} \frac{\cos(f_+(s,t)) - \cos(s) \cos(t)}{\sin(s) \sin(t)}$, so $s=t_s^+$. Analogously, we get $s=\inf \{t:f_-(s,t)>0\}$, in total $f_\pm(s,t)>0$ for $s<t$.

That means that whenever $s_-<t<s_+$ we have that $\alpha(t)\in I(\beta(s_-),\beta(s_+))$. By strong causality the $I(\beta(s_-),\beta(s_+))$ form a neighbourhood basis of the point $\beta(t)$, and $\alpha(t)$ is inside of all of these neighbourhoods. As $X$ is Hausdorff, we get that $\alpha(t)=\beta(t)$ for all $t\in(0,b)$, so $\alpha=\beta$.
\end{proof}

\begin{lem}[\AdSn-Parallel lines are unique]
\label{lem:parallelLinesUnique}
Let $X$ be a strongly causal, timelike geodesically connected \LpLS with $\tau$ continuous on $U$, $\tau(x,x)=0$ for all $x \in X$. Suppose that $X$ has global timelike curvature bounded below by $-1$, let $\alpha:(-\frac{\pi}{2},\frac{\pi}{2})\to U$ be an \AdSn-line and $p\in U$ a point. Then there is at most one \AdSn-parallel line to $\alpha$ through $p$.
\end{lem}
\begin{proof}
We indirectly assume there are two \AdSn-parallel lines to $\alpha$ through $p$, namely $\beta:(-\frac{\pi}{2},\frac{\pi}{2})\to X$ and $\tilde{\beta}:(-\frac{\pi}{2},\frac{\pi}{2})\to X$, with $p=\beta(t_0)=\tilde{\beta}(\tilde{t}_0)$. We know that $t_0=\tilde{t}_0$ and that both $\alpha,\tilde{\beta}$ have the same distance, as there is only one way (up to isometry) of realizing $\alpha$ vertically in $\AdS'$ and putting $p$ at the right distance.

We construct a comparison situation for all three lines at once: We choose the parallel realisation $f$ for $\alpha$ and $\beta$ given by $f(\alpha(s))=(s,0)$ and $f(\beta(s))=(s,c)$, and similarly we choose $\tilde{f}$ for $\alpha$ and $\tilde{\beta}$ given by $\tilde{f}(\tilde{\beta}(s))=(s,-c)$ and $\tilde{f}(\alpha(s)) = (s,0)$, i.e.\ we realise $\beta$ and $\tilde{\beta}$ on opposite sides of $\alpha$. 

We calculate the angle $\omega:=\ma_p(\beta|_{[t_0,\frac{\pi}{2})},\tilde{\beta}|_{[t_0,\frac{\pi}{2})})$: By Prop.\ \ref{pop:angle=comparisonangle}, this is equal to any comparison angle $\tilde{\ma}_p(\beta(s),\tilde{\beta}(t))$ as long as $\beta(s)\ll\tilde{\beta}(t)$ or conversely. 

We know by Lem.\ \ref{lem:communicability} that given $r\in(-\frac{\pi}{2},\frac{\pi}{2})$, we find an $s\in (-\frac{\pi}{2},\frac{\pi}{2})$ such that $\beta(r)\ll\alpha(s)$ and a $t\in (-\frac{\pi}{2},\frac{\pi}{2})$ such that $\alpha(s)\ll\tilde{\beta}(t)$. In particular, for all $r\in(-\frac{\pi}{2},\frac{\pi}{2})$ there is a $t\in(-\frac{\pi}{2},\frac{\pi}{2})$ such that $\beta(r)\ll\tilde{\beta}(t)$. We now look at the law of cosines to estimate the comparison angle $\omega=\tilde{\ma}_p(\beta(r),\tilde{\beta}(t))$: $\cosh(\omega)=\frac{\cos(\tau(\beta(r),\tilde{\beta}(t)))-\cos(r-t_0)\cos(t-t_0)}{\sin(r-t_0)\sin(t-t_0)}\leq \frac{1-\cos(r-t_0)\cos(t-t_0)}{\sin(r-t_0)\sin(t-t_0)}$ which converges to $1$ as $s$ and $t$ both approach the maximal value $\frac{\pi}{2}$. In particular, $\omega=0$. We can now apply the strong causality trick (Lem.\ \ref{lem:strongCausTrick}) to get $\beta=\tilde{\beta}$, so there is only one line parallel to $\alpha$ through $p$.
\end{proof}

Let us now show for $X$ as in the main theorem that asymptotic \AdSn-lines to $\gamma$ constructed in Prop.\ \ref{pop:asymptoticLines} are parallel to $\gamma$.

\begin{lem}
\label{lem:asymptotesAreParallelToGamma}
Let $\alpha: (-\frac{\pi}{2},\frac{\pi}{2}) \to X$ be a complete timelike asymptotic \AdSn-line to $\gamma$. Then $\alpha,\gamma$ are \AdSn-parallel.
\end{lem}
\begin{figure}
\begin{center}
\includegraphics{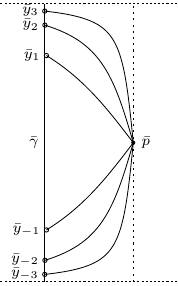}
\end{center}
\caption{Stacking comparison triangles. Every three subsequent $\bar y_k$ together with $\bar{p}$ form two triangles as in Prop.\ \ref{pop:compStack}. }
\label{pop:compStack:figStacking}
\end{figure}
\begin{proof}
By construction, $\alpha$ arises as the limit of timelike maximising segments $\alpha^+_n$ from $p:=\alpha(t_0)$ to $\gamma(r_n)$ and $\alpha^-_n$ from $p$ to $\gamma(t_n)$, where $r_n \to -\frac{\pi}{2}$ and $t_n\to\frac{\pi}{2}$. We will use the stacking principle (cf.\ Prop.\ \ref{pop:compStack}) to show that $\alpha$ and $\gamma$ are parallel. Indeed, the triangles corresponding to $\Delta(\gamma(r_n),p,\gamma(t_n))$ stack in $\AdS'$ (see Figure \ref{pop:compStack:figStacking}), and we have that the sides corresponding to $\alpha^+_n$ and $\alpha^-_n$ in $\AdS'$ converge to the vertical line through $\bar{p}$.

We now argue that the map sending $\alpha((-\frac{\pi}{2},\frac{\pi}{2}))$ and $\gamma((-\frac{\pi}{2},\frac{\pi}{2}))$ to the corresponding limit lines $\bar{\alpha}$ and $\bar{\gamma}$ is a parallel realisation. 

For any $t,s$, consider $\tau(\alpha(t),\gamma(s)) = \lim_n \tau(\alpha_n(t),\gamma(s))$. For $n$ so large that $-r_n < s < r_n$, $\gamma(s)$ is part of the triangle $\Delta(\gamma(-r_n),p,\gamma(r_n))$. From Cor.\ \ref{cor:sides_equal} we conclude that $\alpha_n(t) \ll \gamma(s)$ if and only if $\overline{\alpha}_n(t) \ll \overline{\gamma}(s)$, and $\tau(\alpha_n(t),\gamma(s)) = \overline{\tau}(\overline{\alpha}_n(t),\overline{\gamma}(s))$, where the bars denote the corresponding points on the comparison side. This shows that the realization of $\alpha$ as $\bar{\alpha}$ and $\gamma$ as $\bar{\gamma}$ is a parallel realisation.
\end{proof}

\begin{lem}[Asymptotes are asymptotes through any point on them] \label{lem:shiftingAsymptotes}
Let $\alpha: (-\frac{\pi}{2},\frac{\pi}{2}) \to X$ be an asymptote to $\gamma$ through $\alpha(t_0)$. Then the asymptote to $\gamma$ through $\alpha(s)$ is  also given by $\alpha$.
\end{lem}
\begin{proof}
Let $\tilde{\alpha}$ be the asymptotic line through $\alpha(s)$. Then the previous result shows that both $\gamma$ and $\alpha$ as well as $\gamma$ and $\tilde{\alpha}$ are parallel and they meet at $\alpha(t_0)=\tilde{\alpha}(t_0)$. Thus, as parallel lines are unique (see Lem.\ \ref{lem:parallelLinesUnique}), we have that $\alpha=\tilde{\alpha}$. 
\end{proof}
\begin{cor}[Asymptotes stay in $I(\gamma)$]
\label{cor:asymptotesStayInI(gamma)}
Asymptotic lines to $\gamma$ from points in $I(\gamma)$ stay in $I(\gamma)$.
\end{cor}
\begin{proof}
Since asymptotic lines to $\gamma$ are parallel to $\gamma$ by Lem.\ \ref{lem:asymptotesAreParallelToGamma}, this readily follows.
\end{proof}

We can piece parallelism and asymptotes together using the following strengthening of Lem.\ \ref{lem:asymptotesAreParallelToGamma}.
\begin{lem}[Two asymptotes are parallel] \label{lem:asymptotesAreParallel}
Let $x,y\in I(\gamma)$ and $\alpha,\beta$ the asymptotes to $\gamma$ through $x$ resp.\ $y$. Then $\alpha,\beta$ are \AdSn-parallel.
\end{lem}
\begin{proof}
We dive into the construction of asymptotes: We first assume $x\ll y$. We set $s_0=b^+(x)$ and $t_0=b^+(y)$, then there are maximisers $\alpha_n$ from $x$ to $\gamma(T_n)$ parametrised in $\tau$-arclength such that $\alpha_n(s_0)=x$  and $\beta_n$ from $y$ to $\gamma(T_n)$ parametrised in $\tau$-arclength such that $\beta_n(t_0)=y$, which converge (pointwise) to the upper parts of $\alpha$ and $\beta$ in $\tau$-arclength parametrization for $T_n \to \infty$, respectively.
Note that $\alpha_n(s_0+\tau(x,\gamma(T_n)))=\beta_n(t_0+\tau(y,\gamma(T_n)))=\gamma(T_n)$.

We try to prove the $c$-criterion (Lem.\ \ref{lem:c-cri}): Clearly, due to our assumptions on $X$, we do not need to consider the null functions. First, we look at $c_{\alpha\beta}(s_0,t_0)$ and $c_{\alpha\beta}(s',t')$ or $c_{\beta\alpha}(s',t')$ for $s_0\leq s'$, $t_0\leq t'$: We require $a:=x=\alpha(s_0)\ll b:=y=\beta(t_0)$. We get points converging to the primed parameters: $a'_n:=\alpha_n(s')\to a':=\alpha(s')$ and $b'_n:=\beta_n(t')\to b':=\beta(t')$ (note that $\alpha_n(s_0)=a$ and $\beta_n(t_0)=b$ anyway). We get a timelike triangle $\Delta_n=\Delta(a, b, c_n:= \gamma(T_n))$ containing the points $(a'_n,b'_n)$. We form a comparison situation for this in $\AdS'$: $\bar{\Delta}_n=\Delta(\bar{a},\bar{b}, \bar{c}_n)$ with comparison points $\bar{a}'_n,\bar{b}'_n$. As $X$ has global curvature bounded below by $K=-1$, we get that $\tau(a'_n,b'_n)\leq\bar{\tau}(\bar{a}'_n,\bar{b}'_n)$ and $\tau(b'_n,a'_n)\leq\bar{\tau}(\bar{b}'_n,\bar{a}'_n)$.

We would now like to let $T_n \to+\frac{\pi}{2}$, so we need control over what the comparison triangle $\bar{\Delta}_n$ converges to. For this, we select which way to realise $\bar{\Delta}_n$ in $\AdS'$. We choose:
\begin{itemize}
\item $\bar{a}=(s_0,0)$ constant in $n$, i.e.\ at the right $t$-coordinate and with $x$-coordinate $0$,
\item $\bar{b}=(t_0,c_{\alpha\beta}(s_0,t_0))$ constant in $n$. Note this has the right $\tau$-distance to $\bar{a}$ by definition of $c_{\alpha\beta}$, and it has the right $t$-coordinate.
\item $\bar{c}_n$ with $\bar{\tau}(\bar{a},\bar{c}_n)=\tau(a,c_n)$ and $\bar{\tau}(\bar{b},\bar{c}_n)=\tau(b,c_n)$. 
\begin{itemize}
\item Note that $\bar{c}_n$ need not be realizable in $\AdS'\subseteq \AdS$, but in $\AdS$ it is.
\item But still, the initial segments of the sides $\bar{\alpha}_n$ connecting $\bar{a}\bar{c}_n$ and $\bar{\beta}_n$ connecting $\bar{b}\bar{c}_n$ are contained in $\AdS'$.
\item Note that the future tip $\bar{c}=(1,0,0)\in\AdS$ of $\AdS'$ satisfies $\bar{\tau}(\bar{a},\bar{c})=\lim_n\tau(a,c_n)$ and $\bar{\tau}(\bar{b},\bar{c})=\lim_n\tau(b,c_n)$.
\item Assuming $c_{\alpha\beta}(s_0,t_0)\neq0$, the function $g:\AdS\to\mb{R}^2$, $g(\bar{p})=(\bar{\tau}(\bar{a},\bar{p}),\bar{\tau}(\bar{b},\bar{p}))$ is continuous near $\bar{c}$ and maps a neighbourhood of $\bar{c}$ to a neighbourhood of $(\frac{\pi}{2}-s_0=\lim_n\tau(a,c_n),\frac{\pi}{2}-t_0=\lim_n\tau(b,c_n))$. In particular, for $n$ large we can assume $\bar{c}_n$ lie in this neighbourhood and $\bar{c}_n\to\bar{c}$.  
\item Assuming $c_{\alpha\beta}(s_0,t_0)=0$, we have that $\bar{b}=(t_0,0)$, and as $\tau(a,c_n)$ stays bounded away from $\pi$, the triangle $\Delta(\bar{a},\bar{b},\bar{c}_n)$ converges (in \AdSn) to a triangle which satisfies the size bounds (thus realizable uniquely up to isometries in $\AdS$, giving one or two realizations given $\bar{a}$ and $\bar{b}$), and $\bar{\tau}(\bar{a},\bar{c}_n)\to \frac{\pi}{2}-s_0$ and $\bar{\tau}(\bar{b},\bar{c}_n)\to \frac{\pi}{2}-t_0$ makes the limit triangle degenerate, thus $\bar{c}_n$ converge to $\bar{c}$ (as this point makes $\bar{a},\bar{b},\bar{c}$ degenerate and is at the right distance).
\item In particular, $\bar{\alpha}_n\to\bar{\alpha}$ and $\bar{\beta}_n\to\bar{\beta}$ in $\tau$-arclength parametrization. Thus the initial segments of $\bar{\alpha}_n$ and $\bar{\beta}_n$ become vertical in the limit.
\end{itemize}
\end{itemize}

In the $\AdS'$ picture, the limit comparison sides have the form $\bar{\alpha}(s)=(s,0)$ and $\bar{\beta}(t)=(t,c_{\alpha\beta}(s_0,t_0))$. By continuity of $\tau$ and $\bar{\tau}$, our curvature assumption gives $\tau(a',b')\leq\bar{\tau}(\bar{a}',\bar{b}')$ in the limit (similarly with arguments flipped). Now we compare the definition of $c_{\alpha\beta}(s',t')$ resp.\ $c_{\beta\alpha}(s',t')$ with the $c$-functions for $\bar{\alpha}$ and $\bar{\beta}$ at the same parameters:
\begin{align*}
c_{\alpha\beta}(s',t')&=\arcosh\left(\frac{\cos(\tau(a',b'))-\sin(s')\sin(t')}{\cos(s')\cos(t')}\right),\\
c_{\bar{\alpha}\bar{\beta}}(s',t')&=\arcosh\left(\frac{\cos(\bar{\tau}(\bar{a}',\bar{b}')-\sin(s')\sin(t')}{\cos(s')\cos(t')}\right)= c_{\alpha\beta}(s_0,t_0),
\end{align*}
whenever defined. The last equality holds because we know $\bar{\alpha}$ and $\bar{\beta}$ are \AdSn-parallel with distance $c_{\alpha\beta}(s_0,t_0)$. Note that $a' \leq b'$ implies $\bar{a}' \leq \bar{b}'$ by the curvature bound and a simple continuity argument, so whenever the first line is defined so is the second. Notice these equations only differ in the $\tau$ term, and we know $\tau(a',b')\leq\bar{\tau}(\bar{a}',\bar{b}')$, so we get $c_{\alpha\beta}(s',t')\geq c_{\alpha\beta}(s_0,t_0)$ whenever the former is defined. 

Similarly, if $b' \leq a'$ we get $c_{\beta\alpha}(s',t')\geq c_{\alpha\beta}(s_0,t_0)$ for all $s'\geq s_0$ and $t'\geq t_0$. We can also do this if $a\gg b$, giving $c_{\alpha\beta}(s',t')\geq c_{\beta\alpha}(s_0,t_0)$ and $c_{\beta\alpha}(s',t')\geq c_{\beta\alpha}(s_0,t_0)$. 

Doing the same construction towards the past, we similarly get $c_{\alpha\beta}(s',t')\geq c_{\alpha\beta}(s_0,t_0)$ and $c_{\beta\alpha}(s',t')\geq c_{\alpha\beta}(s_0,t_0)$ resp.\ $c_{\alpha\beta}(s',t')\geq c_{\beta\alpha}(s_0,t_0)$ and $c_{\beta\alpha}(s',t')\geq c_{\beta\alpha}(s_0,t_0)$ (depending on whether $a\ll b$ or $b\ll a$) for all $s'\leq s_0$ and $t'\leq t_0$.

Now we use Lem.\ \ref{lem:shiftingAsymptotes} to get that the \AdSn-asymptote to $\gamma$ through $\alpha(s)$ is $\alpha$, and similarly the asymptote to $\gamma$ through $\beta(t)$ is $\beta$. In particular, we can use the above argument again for $\alpha(s)$ instead of $x=\alpha(s_0)$ and $\beta(t)$ instead of $y=\beta(t_0)$. The above (to the future, $a\ll b$) then gives:
$c_{\alpha\beta}(s',t')\geq c_{\alpha\beta}(s,t)$ for $s\leq s'$, $t\leq t'$ such that $a\ll b$ and $a'\ll b'$, extending continuously, $c_{\alpha\beta}$ is monotonically increasing where defined. On the other hand, the above (to the past, $a\ll b$) gives: 
$c_{\alpha\beta}(s',t')\geq c_{\alpha\beta}(s,t)$ for $s\geq s'$, $t\geq t'$ such that $a\ll b$ and $a'\ll b'$, extending continuously, $c_{\alpha\beta}$ is monotonically decreasing where defined. Via a two-step process, we see that $c_{\alpha\beta}$ is constant where defined. Similarly, we get that also $c_{\beta\alpha}$ is constant where defined, and that they have the same value.

Thus, the $c$-criterion (Lem.\ \ref{lem:c-cri}) yields that $\alpha$ and $\beta$ are synchronised parallel.
\end{proof}

Thus any two asymptotes to $\gamma$ are parallel. 

\section{Proof of the main result}\label{sec:proof}

Let us summarise what we have shown so far: Let $(X,d,\ll,\leq,\tau)$ be a connected, regularly localisable, globally hyperbolic Lorentzian length space satisfying timelike geodesic prolongation, with proper metric $d$ and global timelike curvature bounded below by $K=-1$ containing an \AdSn-line $\gamma:(-\frac{\pi}{2},\frac{\pi}{2}) \to X$. Then from each point in $I(\gamma) = I^+(\gamma) \cap I^-(\gamma)$ we can construct a (unique) future asymptotic ray to $\gamma$, being timelike and having $b_+$ on them go to $\frac{\pi}{2}$, and a (unique) past asymptotic ray to $\gamma$, being timelike and having $b_-$ on them go to $\frac{\pi}{2}$. $b_+=-b_-$ and future directed and past directed rays from a common point fit together to give a timelike line which is parallel to $\gamma$ and can be parametrized in $\tau$-arclength on $(-\frac{\pi}{2},\frac{\pi}{2})$. The Busemann function selects a "spacelike" slice $\{b^+ = 0\}$ in $X$ containing precisely one point on each of the asymptotes which will provide the metric part of the warped product.

In this section, we first prove that $I(\gamma)$ is a warped product and then that $X=I(\gamma)$, establishing that $X$ is a warped product.

\begin{dfn}[Spacelike slice]
We call the set $S=(b_+)^{-1}(0)$ the \emph{spacelike slice}. For $p,q\in S$ we find the asymptotes $\alpha$, $\beta$ to $\gamma$ through $p$ resp.\ $q$, so $\alpha(0)=p$ and $\beta(0)=q$, then they are \AdSn-parallel. We define $d_S(p,q)$ to be the spacelike distance between $\alpha$ and $\beta$ in the sense of \AdSn-parallel lines, i.e.\ the constant $c$ from the $c$-criterion.
\end{dfn}

\begin{lem}
$(S,d_S)$ is a metric space.
\end{lem}
\begin{proof}
As asymptotes to $\gamma$ are parallel to $\gamma$ and parallel lines are unique, $d_S$ is well-defined.

Let $p,q\in S$. It is obvious from the definition of the distance of two parallel lines that $d_S(p,q)\geq 0$. If $d_S(p,q)=0$, the last step in Lem.\ \ref{lem:strongCausTrick} we get that the asymptotes through $p$ and $q$ are the same curve $\alpha$, so we get that $p=\alpha(0)=q$.

For the triangle inequality, let $p,q,r\in S$. We get asymptotes $\alpha$ through $p$, $\beta$ through $q$ and $\eta$ through $r$, and by Lem.\ \ref{lem:asymptotesAreParallel}, they are pairwise AdSn-parallel. Let $d_1=d_S(p,q)$ and $d_2=d_S(q,r)$ and $d_3=d_S(p,r)$. Then we can make a situation in $\AdS'$ which at the same time is a parallel realization for $\alpha,\beta$ and for $\beta,\eta$: $\bar{\alpha}(t)=(t,0)$, $\bar{\beta}(t)=(t,d_1)$, $\bar{\eta}(t)=(t,d_1+d_2)$. 

Now for any $r,t$ such that $\bar{\alpha}(r)\leq\bar{\eta}(t)$, the distance realizer connecting them crosses $\bar{\beta}$, so we find an $s$ such that $\bar{\alpha}(r)\leq\bar{\beta}(s)\leq\bar{\eta}(t)$. But this means $\alpha(r)\leq\beta(s)\leq\eta(t)$, so they are also causally related in $X$. Let $t'=\inf\{t':\alpha(s)\leq\eta(t')\}$, then $t'\leq t$. Now note that the equation $c_{\alpha+}^N(s)=\arcosh\left(\frac{1-\sin(s)\sin(t)}{\cos(s)\cos(t)}\right)$ is increasing in $t$ if $t>s$, so we know
\begin{equation*}
d_3=\arcosh\left(\frac{1-\sin(s)\sin(t')}{\cos(s)\cos(t')}\right)\leq \arcosh\left(\frac{1-\sin(s)\sin(t)}{\cos(s)\cos(t)}\right)=d_1+d_2
\end{equation*}
as desired, where the last follows as $\bar{\alpha}$ and $\bar{\eta}$ have spacelike distance $d_1+d_2$.
\end{proof}

\begin{dfn}[The warped product map]
We define $f:(-\frac{\pi}{2},\frac{\pi}{2})\times_{\cos} S\to I(\gamma) \subset X$ by $f(s,p)=\alpha_p(s)$, where $\alpha_p$ is the \AdSn-asymptote to $\gamma$ through $p$.
\end{dfn}

\begin{pop}[Local splitting]
\label{thm:localSplitting}
Let $X$ be a connected, regularly localisable, globally hyperbolic Lorentzian length space with proper metric $d$ and global timelike curvature bounded below by $K=-1$ satisfying timelike geodesic prolongation and containing an \AdSn-line $\gamma:(-\frac{\pi}{2},\frac{\pi}{2}) \to X$. Then $I(\gamma) \subset X$ is a causally convex open set that is itself a path-connected, regularly localisable, globally hyperbolic Lorentzian length space of global timelike curvature bounded below by $K=-1$ with the metric, relations and time separation induced from $X$. Moreover, the spacelike slice $S$ is a proper (hence complete), strictly intrinsic metric space, the Lorentzian warped product $(-\frac{\pi}{2},\frac{\pi}{2}) \times_{\cos} S$ (a anti-deSitter suspension) is a path-connected, regularly localisable, globally hyperbolic Lorentzian length space and the splitting map $f:(-\frac{\pi}{2},\frac{\pi}{2}) \times_{\cos} S \to I(\gamma)$ is a $\tau$- and $\leq$-preserving homeomorphism.
\end{pop}
\begin{proof}
First, it is clear that $I(\gamma)$ is path-connected, causally convex in $X$ and has global timelike curvature bounded below by $K=-1$. It is hence causally path-connected since $X$ is and it is trivially locally causally closed. Moreover, if $x \in I(\gamma)$ and $U$ is a regular localising neighbourhood of $x$ in $X$, then $U \cap I(\gamma)$ is a regular localising neighbourhood of $x$ in $I(\gamma)$, hence $I(\gamma)$ is regularly localisable. By causal convexity, the time separation between causally related points in $I(\gamma)$ is achieved as the supremum of lengths of causal curves running between them which have to stay inside $I(\gamma)$. The causal diamonds in $I(\gamma)$ are precisely those in $X$, since they must be contained in $I(\gamma)$. Finally, $I(\gamma)$ is non-totally imprisoning (it inherits this from $X$), thus we have shown all the claims on $I(\gamma)$.

Next, we argue that $f$ is $\tau$- and $\leq$-preserving. For now, denote the causal relation and time separation function in $(-\frac{\pi}{2},\frac{\pi}{2})\times_{\cos}S$ by $\mathrel{\bar{\leq}}$ and $\bar{\tau}$. Let $(s_0,p),(t_0,q)\in (-\frac{\pi}{2},\frac{\pi}{2})\times_{\cos}S$, we need to check $(s_0,p)\mathrel{\bar{\leq}} (t_0,q)\Leftrightarrow f(s_0,p)\leq f(t_0,q)$ and $\bar{\tau}((s_0,p),(t_0,q))=\tau(f(s_0,p),f(t_0,q))$. Let $\alpha$ be the \AdSn-asymptote to $\gamma$ through $p$ and let $\beta$ be the \AdSn-asymptote to $\gamma$ through $q$. Then $\alpha$ and $\beta$ are \AdSn-parallel with distance $d_S(p,q)$, so there is a \AdSn-parallel realisation $\tilde{f}:\alpha\cup\beta\to\AdS'$ defined by $\tilde{f}(\alpha(s))=(s,0)$ and $\tilde{f}(\beta(t))=(t,d_S(p,q))$. In particular, $\alpha(s)\leq_X \beta(t)\Leftrightarrow \tilde{f}(\alpha(s))\leq \tilde{f}(\beta(t))$ and if this is true, $\tau(\alpha(s),\beta(t))=\bar{\tau}(\tilde{f}(\alpha(s)),\tilde{f}(\beta(t)))$. 

But the definition of $\leq$ and $\tau$ in the Lorentzian warped product just asks us to realize the points by $(s,0)$ and $(t,d_S(p,q))$, so the $\tau$-distance and causal relation are the same as in the parallel realization, and thus the same as in $X$, so $f$ is $\tau$- and $\leq$-preserving. $f$ is injective as \AdSn-asymptotes to $\gamma$ are \AdSn-parallel to $\gamma$ and \AdSn-parallel lines don't meet (Lemma \ref{lem:parallelLinesUnique}), and surjective since any point in $I(\gamma)$ lies on an \AdSn-asymptote.

From the discussion above, it is easy to see that $f$ maps timelike (and causal) diamonds in $I(\gamma)$ to timelike (and causal) diamonds in $(-\frac{\pi}{2},\frac{\pi}{2}) \times_{\cos} S$. As both sides are strongly causal, this implies that $f$ is a continuous open bijection, hence a homeomorphism. Since $(-\frac{\pi}{2},\frac{\pi}{2}) \times_{\cos} S$ is always non-totally imprisoning (cf.\ Prop.\ \ref{pop:PropsLwp}.\ref{pop:PropsLwp:nonTotImpr}) and its causal diamonds are compact (as continuous images of compact causal diamonds in $I(\gamma)$), we conclude that $(-\frac{\pi}{2},\frac{\pi}{2}) \times_{\cos} S$ is globally hyperbolic, hence $S$ is proper by Prop.\ \ref{pop:PropsLwp}.\ref{pop:PropsLwp:globHyp}. To see that $S$ is a strictly intrinsic space, fix $p,q \in S$ and connect any two timelike related points on the corresponding asymptotes by a distance realiser in $I(\gamma)$. The image of that distance realiser under $f$ is a continuous distance realiser in $(-\frac{\pi}{2},\frac{\pi}{2}) \times_{\cos} S$, hence by Prop.\ \ref{pop:PropsLwp}.\ref{pop:PropsLwp:geodesics} the projection onto $S$ gives a distance minimiser in $S$ between $p$ and $q$.

Finally, we need to prove the remaining claimed properties of $(-\frac{\pi}{2},\frac{\pi}{2}) \times_{\cos} S$. Path-connectedness is inherited from $I(\gamma)$ via $f$, and warped products are always globally causally closed (cf.\ Prop.\ \ref{pop:PropsLwp}.\ref{pop:PropsLwp:nonTotImpr}). Note that $I(x,y)$ for $x,y \in I(\gamma)$ are regular localising neighbourhoods in $I(\gamma)$. Since $f(I(x,y)) = I(f(x),f(y))$ and the $d$-lengths of causal curves in timelike diamonds can always be uniformly bounded in warped products (cf.\ \ref{pop:PropsLwp}.\ref{pop:PropsLwp:Lip}), timelike diamonds in $(-\frac{\pi}{2},\frac{\pi}{2}) \times_{\cos} S$ are in fact (regular) localising neighbourhoods: For the local time separation, take the restriction of $\bar{\tau}$, and note that maximisers in $I(\gamma)$ (or in any $I(x,y) \subset I(\gamma)$) map to continuous maximisers in $(-\frac{\pi}{2},\frac{\pi}{2}) \times_{\cos} S$, which are always Lipschitz reparametrisable and are hence causal curves (cf.\ Prop.\ \ref{pop:PropsLwp}.\ref{pop:PropsLwp:Lip}).
\end{proof}

\begin{pop}[Globalize the splitting]
\label{pop:globalizeSplitting}
Let $(X,d,\ll,\leq,\tau)$ satisfy the assumptions in Prop.\ \ref{thm:localSplitting} and such that for each pair of points $x\ll z$ in $X$ we find $y \in X$ such that $\Delta(x,y,z)$ is a non-degenerate timelike triangle. Then $X= I(\gamma)$.
\end{pop}
\begin{proof}
Using Thm.\ \ref{thm:localSplitting}, we only need to prove $I(\gamma)=X$. We indirectly assume a point in $X \setminus I(\gamma)$: First, if there is a point $p\in I^+(\gamma)\setminus I^-(\gamma)$, we find a $t$ such that $\gamma(t)\ll p$, and a timelike distance realizer $\alpha_t$ connecting them. Inside a localizing neighbourhood of $p$, we have a point $p\ll p_+$ and set $\varepsilon=2\tau(p,p_+)>0$. $\alpha_t$ is initially contained in $I(\gamma)$, so we find a point $\alpha_t(s_t)\in\partial I(\gamma)$, i.e.\ in $\partial I^-(\gamma)$. 

By Lem.\ \ref{lem:causalBdryAtPi} we get that for all $\varepsilon>0$ there is $\delta>0$ small enough such that $\tau(\gamma(r),\alpha_t(s_t-\delta))\to\pi-\varepsilon$ as $r\to-\frac{\pi}{2}$. In particular, $\tau(\gamma(r),p_+)>\pi$. By the synthetic Bonnet-Myers theorem \ref{thm: lor meyers}, we know that this forces $\tau(\gamma(r),p_+)=+\infty$, contradicting finiteness of $\tau$ (see \cite[Thm.\ 3.28]{saemLorLen}).

The situation where there is a point $p\in I^-(\gamma)\setminus I^+(\gamma)$ is analogous. 

Finally, assume that all points in $X \setminus I(\gamma)$ are in neither $I^+(\gamma)$ nor in $I^-(\gamma)$. Take any such point $p$, then by strong causality, we have $p_-\ll p\ll p_+$. Then also $p_-,p_+$ are neither in $I^+(\gamma)$ nor in $I^-(\gamma)$. Thus, $I(p_-,p_+)\subseteq X \setminus I(\gamma)$ is an open neighbourhood of $p$. This works for any $p$, so both $I(\gamma)$ and $X \setminus I(\gamma)$ are open, contradicting connectedness of $X$.
\end{proof}

\begin{thm}\label{thm:globalsplitting}
Let $(X,d,\ll,\leq,\tau)$ be a connected, regularly localisable, globally hyperbolic Lorentzian length space with proper metric $d$ and global timelike curvature bounded below by $K=-1$ satisfying timelike geodesic prolongation and containing a $\tau$-arclength parametrized distance realizer $\gamma:(-\frac{\pi}{2},\frac{\pi}{2}) \to X$. Assume that for each pair of points $x\ll z$ in $X$ we find $y \in X$ such that $\Delta(x,y,z)$ is a non-degenerate timelike triangle.

Then there is a proper (hence complete), strictly intrinsic metric space  $S$ such that the Lorentzian warped product $(-\frac{\pi}{2},\frac{\pi}{2}) \times_{\cos} S$ is a path-connected, regularly localisable, globally hyperbolic Lorentzian length space and there is a map $f:(-\frac{\pi}{2},\frac{\pi}{2}) \times_{\cos} S \to I(\gamma)$ which is a $\tau$- and $\leq$-preserving homeomorphism.
\end{thm}
\begin{proof}
With Prop.\ \ref{pop:globalizeSplitting}, this follows from Prop.\ \ref{thm:localSplitting} immediately.
\end{proof}

Recall the notion of Cauchy sets and (Cauchy) time functions on Lorentzian pre-length spaces from \cite[Sec.\ 5.1]{timeFncLLS}: A Cauchy set is any subset that is met exactly once by doubly inextendible causal curves, and a Cauchy time function is a continuous function $t:X \to \R$ such that $x < y$ implies $t(x) < t(y)$ and the image under $t$ of any doubly inextendible causal curve is all of $\R$.

\begin{cor}\label{cor:formCauchySets} Let $(X,d,\ll,\leq,\tau)$ satisfy the assumptions of Theorem \ref{thm:globalsplitting} and let $f:(-\frac{\pi}{2},\frac{\pi}{2}) \times_{\cos} S \to X$ be the splitting. Then the sets $S_t:=f(\{t\} \times S)$ are Cauchy sets in $X$ that are all homeomorphic to $S$. Moreover, let $\varphi:(-\frac{\pi}{2},\frac{\pi}{2})\to\R$ be a monotonically increasing bijection, then the map $\varphi\circ pr_1 \circ f^{-1}$ is a Cauchy time function. Moreover, all Cauchy sets in $X$ are homeomorphic to $S$.
\end{cor}
\begin{proof}
As $\{t\} \times S$ is acausal in $(-\frac{\pi}{2},\frac{\pi}{2}) \times_{\cos} S$, no causal curve can meet it twice. Next we argue that any doubly inextendible causal curve $\alpha$ meets $S_t$: Suppose that $\alpha$ does not meet $\{t\} \times S$, so w.l.o.g.\ we may assume that $\alpha \subset I^+(\{t\} \times S)$ (as $X=I^-(\{t\} \times S) \cup \{t\} \times S \cup I^+(\{t\} \times S)$ due to the splitting, and this union is disjoint). Let $t_0 \in (a,b)$, then $\alpha((a,t_0]) \subset J^-(\alpha(t_0)) \cap J^+(\{t\} \times S)$ which is easily seen to be a compact set by considering the corresponding situation in $(-\frac{\pi}{2},\frac{\pi}{2}) \times_{\cos} S$. But this is a contradiction, since $X$ is non-totally imprisoning. Hence $\{t\} \times S$ is a Cauchy set. They are obviously homeomorphic to $S$.

Now for $pr_1 \circ f^{-1}$: It is clearly a time function. Now let $\alpha:(a,b) \to X$ be a doubly inextendible causal curve, we need to show that $pr_1 \circ f^{-1} \circ \alpha((a,b)) = (-\frac{\pi}{2},\frac{\pi}{2})$. Suppose not, so there is some time value $t_0$ that is not attained. W.l.o.g.\ suppose $t_0 \geq 0$, so the image is contained in $(-\frac{\pi}{2},t_0]$. Similarly to before, this would imply that $\alpha|_{[t_0,b)}$ is contained in the compact set $J^+(\alpha(t_0)) \cap J^-(S_{t_0})$, a contradiction.

Finally, let $C$ be any Cauchy set in $X$. Then the projection $C \to S$ (via the splitting) is continuous. Its inverse is given by sending each $p \in S$ to the unique point on $\alpha_p$ meeting $C$. This is a continuous map since $C$ is achronal. This shows that $C$ is homeomorphic to $S$.
\end{proof}

Note that in general, Cauchy sets in globally hyperbolic Lorentzian pre-length spaces need not be homeomorphic, as \cite[Ex.\ 5.7,\ Ex.\ 5.8]{timeFncLLS} show.

\begin{cor}\label{cor:SAlexCurv}
Let $(X,d,\ll,\leq,\tau)$ satisfy the assumptions of Thm.\ \ref{thm:globalsplitting}. Then $(S,d_S)$ has Alexandrov curvature bounded below by $-1$.
\end{cor}
\begin{proof}
By Rem.\ \ref{rem:contvsLipschitztriangles}, $(-\frac{\pi}{2},\frac{\pi}{2})\times_{\cos}S$ has timelike curvature bounded below by $-1$. Note that $(-\frac{\pi}{2},\frac{\pi}{2})\times_{\cos}\mathbb{M}^2(-1)$ is a part of $2+1$-dimensional anti-deSitter space (compare to Lem.\ \ref{lem:AdS'}) and thus has timelike curvature bounded above and below by $K=-1$, so we can use \cite[Thm.\ 5.7]{generalizedCones} to get the desired curvature bound on $S$.
\end{proof}

As each strongly causal spacetime can be regarded as a regular Lorentzian length space, we can extend the main result to spacetimes:
\begin{cor}\label{cor:spacetimes}
Let $(M,g)$ be a connected globally hyperbolic spacetime of dimension $n\geq 2$ with smooth timelike sectional curvature bounded above\footnote{This might seem the wrong direction, but this is due to the signature of Lorentzian metrics.} by $K=-1$ and containing a timelike distance realizer of length $\pi$. Furthermore assume along each timelike distance realizer, there are no conjugate points of degree $n-1$. Then there is a spacelike Cauchy surface $S$ in $M$, endowed by a metric from the Riemannian metric $g|_S$ and a map $f:(-\frac{\pi}{2},\frac{\pi}{2})\times_{\cos} S\to M$ which is $\tau$-preserving and a $C^1$ diffeomorphism, restricting to the identity $\{0\}\times S\to S$.
\end{cor}
\begin{proof}
We can regard our strongly causal spacetime $M$  as a connected regular Lorentzian length space $X$, see \cite[Ex.\ 3.24]{saemLorLen}, and when choosing the background metric from a complete Riemannian metric, $d$ will be proper. We use \cite[Thm.\ 3.2]{smoothAndSynthetic} to get that $M$ having smooth timelike sectional curvature bounded above by $K=-1$ implies $X$ having synthetic timelike sectional curvature bounded below by $K=-1$. 
For the non-degeneracy condition: Let $x\ll z\in M$ and connect them by a geodesic $\alpha$ and set $m=\alpha(\frac12)$ to be the midpoint. As $M$ is $n$-dimensional, we find $n-1$ $C^1$ spacelike curves $\beta_k$ starting in $x$ all orthogonal to each other, and look at whether $\tau(x,z)=\tau(x,\beta_k(s))+\tau(\beta_k(s),y)$ for $s$ small enough that they are timelike related. If this holds for small enough $s$, the concatenation of a distance realizer from $x$ to $\beta_k(s)$ with a distance realizer from $\beta_k(s)$ to $z$ is a distance realizer, hence a geodesic variation, making $x,z$ conjugate points. By assumption, this cannot happen $n-1$ times, so set $y=\beta_k(s)$ for the other $k$ and $s$ small enough, then this will satisfy the required inequality.

Now we can apply Thm.\ \ref{thm:globalsplitting} to get $S$ and the splitting map $f$. Identify $S$ with $f(\{0\}\times S)$, then by Cor.\ \ref{cor:formCauchySets}, $S$ can be seen as a Cauchy set in $M$, and as such will be a Lipschitz hypersurface. 

To see it is a spacelike submanifold, take a point $p\in S$ and a cylindrical coordinate system $\varphi$ around $p$. Take the \AdS-asymptote $\alpha$ to $\gamma$ through $p$ (making $p=\alpha(0)$). Take the two $\tau$-level sets $S_-^{\lambda}=\{q:\tau(q,\alpha(\lambda))=\tau(p,\alpha(\lambda))\}$ and $S_+^{\lambda}=\{q:\tau(\alpha(-\lambda),q)=\tau(\alpha(-\lambda),p)\}$. We have that $\varphi(S_{\pm}^{\lambda})$ and $\varphi(S)$ are graphs of functions $s_{\pm}^{\lambda}$, $s$ in the $t$-coordinate, with $s_-^{\lambda}$ increasing in $\lambda$, $s_+^{\lambda}$ decreasing in $\lambda$ and  $s_-^{\lambda}\leq s\leq s_+^{\lambda}$. As $\alpha$ does not have any conjugate points and $\alpha(-\lambda),p,\alpha(\lambda)$ form a degenerate situation, $s_{\pm}^{\lambda}$ are $C^1$ and touching in $\varphi(p)$, making also $s$ differentiable in $\varphi(p)$. As for $\lambda\to\frac{\pi}{2}$, $(s_{+}^{\lambda}-s_{-}^{\lambda})''\to 0$ (they actually approach $s$), $s$ is even twice differentiable in $\varphi(p)$.
As this holds for all $p$, $S$ is a $C^2$ submanifold.

To see $d:S\times S\to\mb{R}$ is the same as the distance induced by the Riemannian metric $g|_S$, first note that in the Lorentzian warped product comparison space for $\cos$, at $S_0=\{(t=0,x):x\in\mb{R}\}$ we have that the distance induced by the Riemannian metric $g|_{S_0}$ is just $d(x,y)=|x-y|$. Now one sees the length of a vector $v\in T_p S$ via the two-dimensional subset generated with $v$ and the \AdS-asymptote through $p$, being $\tau$-isometric to the Lorentzian warped product comparison space for $\cos$, see also \cite{BEE96}. In particular, lengths of curves are the same, and as $S$ is strictly intrinsic we have $d= d_g$. This promotes $(-\frac{\pi}{2},\frac{\pi}{2})\times_{\cos} S$ to a $C^1$ spacetime.

Now just notice that $f^{-1}$ is the normal exponential map of $S$. As $S$ has no focal points in $M$ ($\tau(p,\alpha(t))$ stays maximal among $\tau(S,\alpha(t))$ for $\alpha$ an asymptote to $\gamma$ through $p$), $f^{-1}$ is a $C^1$ diffeomorphism.
\end{proof}
\begin{rem}
For higher differentiability of $S$, one can note that $s_{+}^{\lambda}-s_{-}^{\lambda}\to 0$ in all derivatives.
\end{rem}

\noindent{\em Acknowledgements.} Tobias Beran acknowledges the support of the University of Vienna. I would like to thank Argam Ohanyan for inspiration and Felix Rott for helpful discussions and valuable feedback. Furthermore, I have to thank Willi Kepplinger for help with the smooth case and the referee for many helpful comments and improving the readability.
\nocite{saemLorLen,BBI,generalizedCones,semiRiemAlexBounds,AngLLS}
\nocite{harrisTriangleComparison,triSplitting,alexPatchAndBonnet}
\nocite{didierAngles,timeFncLLS}
\bibliography{references} 
\bibliographystyle{plain}
\end{document}